\documentclass[11pt]{article}
\usepackage[margin=1in]{geometry}
\usepackage[utf8]{inputenc}

\usepackage[english]{babel}
\usepackage[T1]{fontenc}
\usepackage{amssymb,amsmath,amstext,amsfonts,amsthm,braket}
\usepackage{mathtools}
\usepackage{verbatim}
\usepackage{relsize}
\usepackage[colorlinks=true, urlcolor=blue, linkcolor=blue, citecolor=blue]{hyperref}
\usepackage{fancyvrb}
\usepackage{tikz-cd}
\usepackage[percent]{overpic}
\usepackage{comment}
\usepackage{algorithm}
\usepackage{algpseudocode}
\usepackage{todonotes}
\usepackage{xcolor}
\usepackage{cleveref}
\newcommand\independent{\protect\mathpalette{\protect\independenT}{\perp}}
\def\independenT#1#2{\mathrel{\rlap{$#1#2$}\mkern2mu{#1#2}}}
\usepackage{comment}
\usepackage[normalem]{ulem}
\usepackage{blkarray}
\usepackage{subcaption}
\usepackage{enumitem}

\numberwithin{equation}{section}
\theoremstyle{plain}
\newtheorem*{theorem*}{Theorem}
\newtheorem{theorem}{Theorem}
\numberwithin{theorem}{section}
\newtheorem{proposition}[theorem]{Proposition}
\newtheorem{lemma}[theorem]{Lemma}
\newtheorem{corollary}[theorem]{Corollary}

\theoremstyle{definition}
\newtheorem{definition}[theorem]{Definition}
\newtheorem{remark}[theorem]{Remark}
\newtheorem{example}[theorem]{Example}

\newcommand{\M}{\mathcal{M}}

\newcommand{\C}{\mathbb{C}}

\newcommand{\Z}{\mathbb{Z}}
\newcommand{\R}{\mathbb{R}}
\newcommand{\N}{\mathbb{N}}
\newcommand{\indep}{\perp \!\!\! \perp}
\newcommand{\Sec}{\mathrm{Sec}}
\newcommand{\rank}{\mathrm{rk}}

\DeclareMathOperator{\glo}{global}
\DeclareMathOperator{\MLdeg}{MLdeg}
\DeclareMathOperator{\Mixt}{Mixt}
\newcommand{\rk}{\text{rank}}

\definecolor{MyBlue}{RGB}{0,101,189} 
\definecolor{MyRed}{RGB}{234, 114, 55} 
\definecolor{MyGreen}{RGB}{162,173,0}

\newcommand{\bfa}{\mathbf{a}}

\newcommand{\bfd}{\mathbf{d}}
\newcommand{\bfh}{\mathbf{h}}
\newcommand{\bfi}{\mathbf{i}}

\newcommand{\bfp}{\mathbf{p}}

\newcommand{\bfx}{\mathbf{x}}
\newcommand{\bfy}{\mathbf{y}}

\newcommand{\bfabar}{\mathbf{\bar{a}}}
\newcommand{\bfdbar}{\mathbf{\bar{d}}}
\newcommand{\bfhbar}{\mathbf{\bar{h}}}

\newcommand{\footremember}[2]{%
    \footnote{#2}
    \newcounter{#1}
    \setcounter{#1}{\value{footnote}}%
}

\title{Mixtures of Discrete Decomposable \\ Graphical Models}

\author{Yulia Alexandr\footremember{yulia}{University of California, Los Angeles; \href{mailto:yulia@math.berkeley.edu}{yulia@math.ucla.edu}}
\and Jane Ivy Coons\footremember{jane}{St John's College, University of Oxford; \href{mailto:jane.coons@maths.ox.ac.uk}{jane.coons@maths.ox.ac.uk}}  
\and Nils Sturma\footremember{nils}{Technical University of Munich, Germany; \href{mailto:nils.sturma@tum.de}{nils.sturma@tum.de}}}

\date{ }

\begin{document}
\maketitle
\begin{abstract}
We study mixtures of decomposable graphical models, focusing on their ideals and dimensions. For mixtures of clique stars, we characterize the ideals in terms of ideals of mixtures of independence models. We also give a recursive formula for their ML degrees.  Finally, we prove that second secant varieties of all other decomposable graphical models have the expected dimension. \looseness=-1
\end{abstract}

\section{Introduction}
Undirected graphical models are  statistical models that capture complex relationships between a collection of random variables. Each random variable is represented by a node in the graph, while the edges specify conditional dependence relations between the random variables. The assumption that statistical relationships have a graphical representation is ubiquitous in applications, as the associated model adopts a natural parametrization which can be read from the structure of the underlying graph. Graphical models are therefore widely used in statistics~\cite{cowell1999probabilistic, edwards2012introduction, hojsgaard2012graphical, whittaker1990graphical} and beyond~\cite{koller2009probabilistic, sinoquet2014probabilistic}. \looseness=-1

In the case where all variables are observable, graphical models are well-studied, see~\cite{lauritzen1996graphical,maathuis2019handbook, sullivant2018algebraic}. However, in practice, one often encounters latent variables that are impossible or infeasible to measure directly which can lead to geometric challenges. Mixture models are among the most accessible latent variable models. Given a statistical model $\mathcal{M}$, the $r$-th mixture model consists of distributions \looseness=-1
\[
\text{Mixt}^r(\mathcal{M}):=\{\pi_1 p^1 + \cdots + \pi_r p^r: (\pi_1,\ldots,\pi_r) \in \Delta_{r} \text{ and } p^1,\ldots,p^r \in \mathcal{M}\},
\]
where $\Delta_{r}$ denotes the  ($r-1$)-dimensional probability simplex.
An interpretation of mixture models is that there are several subpopulations in a population. The proportion of the $i$-th subpopulation in the full population is $\pi_i$ and the $i$-th subpopulation follows a distribution $p^i \in \mathcal{M}$. When observing an individual from the population, we do not know to which subpopulation the individual belongs~\cite[Chapter 14]{sullivant2018algebraic}. Said differently, mixture models feature one latent variable that can be interpreted as representing the unknown membership to the subpopulation. Latent variable models for discrete variables are applied in many disciplines such as psychometrics~\cite{lee2021estimating}, social sciences~\cite{collins2013latent} or phylogenetics~\cite{allman2008phylogenetic, semple2003phylogenetics, zwiernik2016semialgebraic}.

In this paper, we study mixtures of undirected graphical models for discrete random variables in the case where the underlying graph is decomposable. We consider the vanishing ideals and dimensions of these models. The dimension is of particular importance because it can be used to deduce identifiability results. Indeed, the Zariski closure of the mixture model is the $r$th secant variety of the original graphical model. This secant variety may have the expected dimension $r \dim(\mathcal{M}) + r -1$ or it could be defective (Definition \ref{def:expected-dim-defective}). 
If the secant variety has the expected dimension, then the map sending the tuple $(\pi,p^1,\dots,p^r) \in \Delta_r \times \mathcal{M}^r$ to the corresponding distribution in $\Mixt^r(\mathcal{M})$ is generically finite-to-one. 
Hence the parameters $(\pi,p^1,\dots,p^r)$ are \emph{locally identifiable}. 
If the secant variety is defective, then they are non-identifiable.

The simplest model one might consider is the full independence model. It is equivalent to the graphical model where the set of edges is empty. Discrete random variables $X_1,\dots,X_k$ on $n_i$  states respectively are mutually independent if and only if the $n_1 \times \dots \times n_k$ tensor of their joint probabilities has rank~$1$. By definition, the $r$-mixture of the independence model is thus given by the intersection of the probability simplex with the set of tensors of nonnegative rank at most~$r$. This model is also known as the latent class model; see e.g.~\cite{allman2019maximum}. For $r=2$, the geomerty of tensors of nonnegative rank is well-understood; see~\cite{allman2015tensors}. Tensors of nonnegative rank $r \geq 3$ are more complicated: $r$-mixtures of two independent random variables are studied in~\cite{kubjas2015fixed} and $r$-mixtures of multiple independent binary random variables are studied in~\cite{seigal2018mixtures}.  In general, an ideal-theoretic description of  tensors of nonnegative rank at most $r$ for arbitrary $r$ and arbitrary tensor dimension is still unknown. This ideal-theoretic description is the same as for tensors of rank at most $r$ of the same format, since the two sets have the same Zariski closure. 
On the other hand, the dimension of the set of tensors of rank $r$ is extensively studied; see~\cite{landsberg2012tensors} for an overview of known results. Notably, they may or may not have the expected dimension, depending on the value of $r$ and the shape of the tensor.

The present paper goes beyond the independence model by allowing for dependencies specified by an undirected decomposable graph. It is organized as follows. In Section \ref{sec:preliminaries}, we introduce undirected graphical models, mixture models, and secant varieties. In Section \ref{sec:clique stars}, we study the ideals of $r$-mixtures of graphical models. For models corresponding to a special family of decomposable graphs, which we call clique stars, we completely characterize the ideals of their $r$-mixtures in terms of certain tensors of rank~$r$ in Theorem~\ref{thm:clique star-ideals}. Consequently, their dimension depends on the dimension of sets of tensors of rank $r$. In particular, they might be defective. In Section \ref{sec:dimension}, we focus on the dimension of $2$-mixture models of general decomposable graphs. Theorem~\ref{thm:ExpectedDimension} proves that they always have expected dimension if the graph is not a clique star. Finally, Section \ref{sec:ml-degree} gives a recursive formula for the ML degree of mixtures of clique stars.

\section{Preliminaries}\label{sec:preliminaries}
In this section, we formally define undirected graphical models and mixture models. We collect results from the literature on the vanishing ideal and dimensions of graphical models and introduce the tools used in the proofs of our results.

\subsection{Undirected Graphical Models}

For each $n$, denote by $\Delta_n$ the $(n-1)$-dimensional open probability simplex in $\R^n$; that is, 
$$
\Delta_n = \{ \bfp \in \R^n \mid p_i > 0 \text{ for all $i$ and } p_1 + \cdots + p_n = 1 \}.
$$
Undirected graphical models are parameterized subsets of the probability simplex. To define them, let $G=(V, E)$ be an undirected graph where the set of nodes corresponds to the indexing set of a discrete random vector $X=(X_v)_{v \in V}$. We denote by $[d_v]=\{1,\ldots,d_v\}$ the set of values taken by the discrete random variable $X_v$, where $d_v$ is a positive integer. Thus, the random vector $X$ takes values in the state space $\mathcal{R}=\prod_{v \in V} [d_v]$. For a subset $A \subseteq V$, we denote the corresponding subvector by $X_A=(X_v)_{v \in A}$ and its state space by $\mathcal{R}_A=\prod_{v\in A} [d_v]$, where the number of states is given by $d_A = \prod_{v\in A} d_v$. 

Graphical models are parametrized using maximal cliques of the underlying graph $G$. A \textit{clique} $C \subseteq V$ is a set of nodes such that $\{v,w\} \in E$ for every pair $v,w \in C$. A clique is \textit{maximal} if it is containment-maximal among the set of all cliques in $G$. We denote the set of maximal cliques by~$\mathcal{C}(G)$. The following definition is from \cite[Chapter 13]{sullivant2018algebraic}, also see \cite{lauritzen1996graphical}.

\begin{definition}
Let $G$ be an undirected graph. The parameter space of the graphical model specified by $G$ is the set of all tuples of the form $\theta = \Big(\theta_{i_C}^{(C)}\Big)_{i_C \in \mathcal{R}_C, C \in \mathcal{C}(G)}$. Let $m$ be the dimension of such a vector so that $m = \sum_{C \in \mathcal{C}(G)} \prod_{v \in C} d_v$. Define the monomial map $\phi^G: \mathbb{C}^m \rightarrow \mathbb{C}^{|\mathcal{R}|}$ by
\[
    \phi^G_i(\theta) = \prod_{C \in \mathcal{C}(G)} \theta_{i_C}^{(C)}.
\]
The \emph{(parametrized) discrete graphical model} associated to $G$ is given by 
\[
    \M_G := \phi^G(\mathbb{C}^m) \cap \Delta_{|\mathcal{R}|}.
\]
\end{definition}

We denote the coordinates of $\C^{|\mathcal{R}|}$ by $p_{i_1\dots\ i_n}$ where each $i_v \in [d_v]$.  The parameters $\theta_{i_C}^{(C)}$ represent the joint probability that the random variables in the clique $C$ have states $i_C$.

\tikzset{
      every node/.style={circle, inner sep=0.3mm, minimum size=0.3cm, draw, thick, black, fill=white, text=black},
      every path/.style={thick}
}

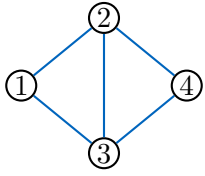
\begin{figure}[t]
\centering
\begin{tikzpicture}[align=center]
    \node[] (1) at (-1.1,0) {$1$};
    \node[] (2) at (0,0.9) {$2$};
    \node[] (3) at (0,-0.9) {$3$};
    \node[] (4) at (1.1,0) {$4$};
    
    \draw[MyBlue] (1) edge (2);
    \draw[MyBlue] (1) edge (3);
    \draw[MyBlue] (2) edge (4);
    \draw[MyBlue] (3) edge (4);
    \draw[MyBlue] (2) edge (3);
\end{tikzpicture}
\caption{Diamond graph.}
\label{fig:diamond-graph}
\end{figure}

\begin{example}\label{ex:diamond}
Consider the diamond graph in Figure $\ref{fig:diamond-graph}$ with nodes $V=[4]$. It has two maximal cliques $C_1 = \{1, 2, 3\}$ and $C_2=\{2,3,4\}$. If $X_1$ and $X_2$ are binary and $X_3$ and $X_4$ are ternary, we have $\mathcal{R} = [2] \times [2] \times [3] \times [3]$. Summing over the two cliques, we see that $m = 2 \cdot 2 \cdot 3 + 2 \cdot 3 \cdot 3 = 30$. So the parametrization of the graphical model is given by
    $
        p_{i_1 i_2 i_3 i_4} = \theta_{i_1 i_2 i_3}^{(C_1)} \theta_{i_2 i_3 i_4}^{(C_2)},
    $
    for all $i_1, i_2 \in [2]$ and $i_3, i_4 \in [3]$.
\end{example}

Denote by $\mathbb{C}[p]=\mathbb{C}[p_i | i \in \mathcal{R}]$ the polynomial ring in $|\mathcal{R}|$ indeterminates and let $I_G \subseteq \mathbb{C}[p]$ be the vanishing ideal of $\phi^{G}(\mathbb{C}^m)$. Since $\phi^{G}$ is a monomial map, $I_G$ is a \textit{toric ideal}, which means it is prime and generated by binomials. 

Alternatively to being defined as the image of a parametrization, undirected graphical models can also be defined as the set of distributions that satisfy certain conditional independence statements. For pairwise disjoint $A,B,S\subset V$ with $A$ and $B$ nonempty, we say that $S$ \textit{separates} $A$ and $B$ if
no pair of nodes $a\in A$ and $b\in B$ lie in the same connected component of the subgraph induced by $V\setminus S$. We will write $A\perp_{S} B\in G$ to denote that $S$ separates $A$ and $B$ in $G$; this is called a \emph{separation statement}. We say that a separation statement $A\perp_{S} B$ is \emph{saturated} if $A \cup S \cup B = V$. The \textit{global Markov property} of $G$ is the collection of conditional independence statements
$$\glo(G)=\{X_A\independent X_B\,|\,X_S: A\perp_{S} B\in G\}.$$
Since the parametrized model $\M_G$ is defined to be a subset of the open probability simplex $\Delta_{|\mathcal{R}|}$, all probabilities are positive and the Hammersley-Clifford theorem \cite[Theorem 13.2.3]{sullivant2018algebraic} verifies that the set of discrete probability distributions satisfying all conditional independence statements in $\glo(G)$ is indeed the same as $\M_G$.  

\begin{example}
    Consider the diamond graph in Figure $\ref{fig:diamond-graph}$.  Note that the set $C=\{2,3\}$ separates the sets $A=\{1\}$ and $B=\{4\}$ in the graph. Thus, for all discrete random vectors $X=(X_v)_{v \in V}$ with distribution $p \in \mathcal{M}_G$, the conditional independence statement $X_1 \indep X_4 | (X_2,X_3)$ holds. 
\end{example}

By~\cite[Proposition 4.1.6]{sullivant2018algebraic}, the conditional independence statement $X_A\independent X_B\,|\,X_S $  is satisfied if and only if for all $i_A,j_A\in\mathcal{R}_A$, $i_B,j_B\in\mathcal{R}_B$ and $i_S\in\mathcal{R}_S$, we have
\begin{align}\label{eq:ci-relations-binomials}
    p_{i_Ai_Bi_S+}\cdot p_{j_Aj_Bi_S+}-p_{i_Aj_Bi_S+}\cdot p_{j_Ai_Bi_S+}=0,
\end{align}
where
    \[
        p_{i_Ai_Bi_S+} := \sum_{j_{V\setminus (A \cup B \cup S)} \in \mathcal{R}_{V\setminus (A \cup B \cup S)}} p_{i_Ai_Bi_Si_{V\setminus (A \cup B \cup S)}}.
    \]

For example, consider the diamond graph from Example \ref{ex:diamond} with $A = \{1\}$, $B = \{4\}$ and $S = \{2,3\}$. Choosing $i_A = 1,$ $ j_A = 2$, $i_B = 2$, $j_B = 3$ and $i_S = 13$, we obtain the polynomial  $p_{1132}p_{2133} - p_{1133}p_{2132}$ such a polynomial. The ideal $I_{\glo(G)}$ is generated by all polynomails of the form~\eqref{eq:ci-relations-binomials}. This ideal may, in general, not be equal to the prime vanishing ideal $I_G$ of the model $\M_G$, but the containment $I_{\glo(G)}\subseteq I_G$ always holds; compare to  \cite{geiger2006onthetoric}. We summarize this observation in the next lemma that will be useful in Section~\ref{sec:clique stars}.
For a fixed state $j_S \in \mathcal{R}_S$, let $M_{j_S; A\independent B}$ be the $d_A\times d_B$ matrix with entries $p_{i_Ai_Bj_S+}$, where the rows are indexed by $i_A \in \mathcal{R}_A$, and the columns are indexed by $i_B \in \mathcal{R}_B$.

\begin{lemma}\label{lem:2x2-minors}
Suppose $A\perp_{S} B\in G$. Then, the prime ideal $I_G$ contains all $2\times 2$ minors of the matrix $M_{j_S; A\independent B}$ for every fixed value $j_S\in\mathcal{R}_S$.  
\end{lemma} 

\begin{example}
    Consider again the diamond graph from Example \ref{ex:diamond} with $A = \{1\}$, $B = \{4\}$ and $S = \{2,3\}$. Fixing $j_S = 13$, and assuming that $X_1$ and $X_4$ are ternary random variables, we obtain the matrix
    \[
        M_{j_S; A\independent B} = \begin{pmatrix}
            p_{1131} & p_{1132} & p_{1133} \\
            p_{2131} & p_{2132} & p_{2133} \\
            p_{3131} & p_{3132} & p_{3133}
        \end{pmatrix}.
    \]
    Considering the first two rows and the last two columns, we see that the binomial $p_{1132}p_{2133} - p_{1133}p_{2132}$ is in the vanishing ideal $I_G$.
\end{example}

The dimensions of graphical models  were derived in~\cite{hosten2002grobner}. For any graph $G$, the dimension of the model $\mathcal{M}_G$ is
\begin{equation*} 
    \dim(\mathcal{M}_G) = \sum_{\substack{C \subseteq V \text{ clique} \\ C \neq \emptyset}} \,\, \prod_{v \in C} (d_v - 1),
\end{equation*}
where the sum runs over all nonempty cliques in the graph $G$. 

\begin{example}
The diamond graph in Figure \ref{fig:diamond-graph} has $4$ cliques of size 1 (nodes), $5$ cliques of size two (edges) and $2$ cliques of size three (triangles). If all random variables take $3$ states, that is, $d_1 = d_2 = d_3 = d_4 = 3$, it follows that
\[
    \dim(\M_G) = 4 \cdot 2 + 5 \cdot 4 + 2 \cdot 8 = 44.
\]
\end{example}

In the present work, we focus on the family of \emph{decomposable graphs}.  These are defined recursively in the following way.

\begin{definition} \cite[Definition~2.3]{lauritzen1996graphical}
    A graph $G$ is \emph{decomposable} if $G$ is complete or if there exist $A, B, S \subset V$ whose union is all of $V$ such that 
    \begin{itemize}
        \item[(i)] the subgraph induced by $S$ is complete,
        \item[(ii)] $S$ separates $A$ and $B$,
        \item[(iii)] the subgraphs induced by $A$ and $B$ are decomposable.
    \end{itemize} 
\end{definition}

Decomposability is equivalent to several other useful conditions. For example, a graph is decomposable if and only if it is chordal. In Section \ref{sec:dimension}, we make use of the fact that $G$ is decomposable if and only if there exists an order $D_1,\dots,D_k$ on its maximal cliques such that $G$ is obtained by repeatedly applying the clique-sum operation to the list in this order \cite{Dirac1961OnRC}, \cite[Theorem 8.3.5]{sullivant2018algebraic}. 

\subsection{Log-linear Models}
Since graphical models admit a monomial parametrization, they form a subclass of more general \textit{log-linear models}, which correspond algebraically to \textit{toric varieties}. Log-linear models are well-known in the literature and are described, for example, in \cite[Chapter 4]{lauritzen1996graphical} and \cite[Chapter 6]{sullivant2018algebraic}. In this section, we recall their definition and the most important properties.

Every log-linear model is specified by an integer matrix $A\in\mathbb{Z}^{m \times n}$ with the vector of all ones in its rowspan. It is defined as $\M_A = \{ p \in \Delta_n \mid \log p \in \mathrm{rowspan}(A) \}$, where $\log p$ denotes the component-wise logarithm of $p$. 
Relaxing the assumption $p \in \Delta_n$ and taking Zariski closure, we obtain a toric variety, whose vanishing ideal is the kernel of the monomial map
$$\C[p_1,\dots, p_n]  \rightarrow \C[\theta_1,\dots,\theta_m]: 
 p_j  \mapsto  \prod_{i=1}^m \theta_i^{a_{ij}}.$$
Explicitly, this toric ideal is $$I_A= \langle p^{u} - p^{v} \mid u - v \in \mathrm{ker}_\Z (A) \rangle$$ where $p^u$ is the usual shorthand for $\prod_{i=1}^r p_i^{u_i}$~\cite[Corollary 4.3]{sturmfels1996grobner}. 

\begin{proposition}\cite[Proposition 1.2.9]{lectures}
Let $G$ be an undirected graph and $\mathbf{d}\in\mathbb{N}^{|V(G)|}$ be the list of the number of states of the random variable at each node. The undirected graphical model $\M_G$ is the log-linear model specified by the $0/1$~matrix $A_{G,\mathbf{d}}$, whose rows are labeled by 
\[\{\theta^{(C)}_{i_C} | C \in \mathcal{C}(G), i_C \in \mathcal{R}_C \}
\]
and columns are labeled by $\{p_j: j\in\mathcal{R}\}$. The entry $\Big(\theta^{(C)}_{i_C},p_j\Big)$\ is 1 whenever $j_C=i_C$ and 0 otherwise.
\end{proposition}

The matrix $A_{G,\mathbf{d}}$ is known as the \emph{$A$-matrix} of $\M_G$ the structure of the $A$-matrix often reveals insights into the model's properties. For example, the rank of the $A$-matrix is one more than the dimension of the model. To each toric variety, we can also associate a polytope, given by the convex hull of the columns of the $A$-matrix. We will denote this polytope by $\text{conv}(A)$.

\begin{example}
Consider the graphical model given by $P_3 \sqcup P_1$, that is, the path $X_1 - X_2 - X_3$ together with an isolated variable $X_4$. If all four variables are binary, the $A$-matrix of $\M_G$ is of the form \looseness=-1
\[ {\scriptsize
    A_{G,\mathbf{d}} = \begin{blockarray}          {cp{0.3cm}p{0.3cm}p{0.3cm}p{0.3cm}p{0.3cm}p{0.3cm}p{0.3cm}p{0.3cm}p{0.3cm}p{0.3cm}p{0.3cm}p{0.3cm}p{0.3cm}p{0.3cm}p{0.3cm}p{0.3cm}}
         & 1111 & 1112 & 1121 & 1122 & 1211 & 1212 & 1221 & 1222 & 2111 & 2112 & 2121 & 2122 & 2211 & 2212 & 2221 & 2222 \\
        \begin{block}{c[cccccccccccccccc]}
        11\cdot\cdot &1&1&1&1&0&0&0&0 &0&0&0&0&0&0&0&0 \\
        12\cdot\cdot &0&0&0&0&1&1&1&1 &0&0&0&0&0&0&0&0\\
        21\cdot\cdot &0&0&0&0&0&0&0&0 &1&1&1&1&0&0&0&0\\
        22\cdot\cdot &0&0&0&0&0&0&0&0 &0&0&0&0&1&1&1&1\\
        \BAhline
        \cdot11\cdot &1&1&0&0&0&0&0&0 &1&1&0&0&0&0&0&0\\
        \cdot12\cdot &0&0&1&1&0&0&0&0 &0&0&1&1&0&0&0&0\\
        \cdot21\cdot &0&0&0&0&1&1&0&0 &0&0&0&0&1&1&0&0\\
        \cdot22\cdot &0&0&0&0&0&0&1&1 &0&0&0&0&0&0&1&1\\
        \BAhline
        \cdot\cdot\cdot1 &1&0&1&0&1&0&1&0 &1&0&1&0&1&0&1&0\\ 
        \cdot\cdot\cdot2 &0&1&0&1&0&1&0&1 &0&1&0&1&0&1&0&1 \\
        \end{block}
    \end{blockarray}}.
\]
Here, the row blocks correspond to the maximal cliques, i.e., the two edges and the isolated node. Within each block, the rows are labeled by the states of the random variables at each clique. The columns are labeled by the states in $\mathcal{R}$. The associated polytope $\text{conv}(A)$ is a 6-dimensional polytope in $\mathbb{R}^{10}$ with the $f$-vector $(16, 56, 92, 82, 40, 10)$.
\end{example}

\subsection{Mixture Models and Secant Varieties}
Let $\M\subset\Delta_{k}$ be a discrete statistical model. The \textit{$r$-th mixture model} of $\M$, denoted by $\text{Mixt}^r(\M)$, is defined as the set of all linear combinations of any $r$ distributions in $\M$. As in \cite[Chapter 14]{sullivant2018algebraic}, we formally have
\[
    \text{Mixt}^r(\mathcal{M}):=\{\pi_1 p^1 + \cdots + \pi_r p^r: (\pi_1,\ldots,\pi_r) \in \Delta_{r} \text{ and } p^1,\ldots,p^r \in \mathcal{M}\}.
\]
Mixtures of undirected graphical models admit a parametrization induced by the original models. Given a graph $G$, the mixture model $\text{Mixt}^r(\mathcal{M}_G)$ is parametrized as 
\begin{equation} \label{eq:parametrization}
    p_{i} = \sum_{j=1}^r \prod_{C \in \mathcal{C}(G)} \theta_{i_C}^{(j,C)}
\end{equation}
for all $p \in \text{Mixt}^r(\mathcal{M}_G)$. We note that since each $\theta_{i_C}^{(j,C)}$ is allowed to range over $\C$ before intersecting with the probability simplex, it is not necessary to include the mixing parameters $\pi$. As outlined in the introduction, mixture models can be seen as models with one latent variable $X_{\ell}$ with state space $[r]$ such that $P(X_{\ell} = j) = \pi_j$. The index $j$ in the parametrization~\eqref{eq:parametrization} corresponds to the state of this latent variable. Our first observation is that mixtures of graphical models are equal to latent variable \emph{graphical models}, where the latent variable is connected to each observed variable. This result is well-known in the statistics and algebraic statistics communities. We include its proof here for completeness.

\begin{proposition}
    Let $r \geq 1$ be an integer and denote by $X_{\ell}$ a (unobserved) random variable with state space $[r]$. Then, the $r$-mixture graphical model $\text{Mixt}^r(\mathcal{M}_G)$ is equal to the latent variable model corresponding to the graph $G_{\ell}=(V \cup \{\ell\}, E_{\ell})$, where we add one latent node $\ell$ and connect it to each observed node $v \in V$, that is, $E_{\ell}=E \cup \{\{v,\ell\}: v \in V\}$. 
\end{proposition}
\begin{proof}
Let $\mathcal{M}_{G_{\ell}}$ be the fully observed graphical model corresponding to the graph $G_{\ell}$. There is a one-to-one correspondence between the maximal cliques of $G$ and the maximal cliques of $G_{\ell}$. Each maximal clique $C \in \mathcal{C}(G)$ corresponds to the maximal clique $C \cup \{\ell\} \in \mathcal{C}(G_{\ell})$. In particular, node $\ell$ is contained in any maximal clique of the graph $G_{\ell}$. Thus, the model $\mathcal{M}_{G_{\ell}}$ is parametrized as
\[
p_{i} = \prod_{C \in \mathcal{C}(G)} \theta_{i_{C}, i_{\ell}}^{(C)} 
\]
for all $p \in \M_{G_{\ell}}$. Now, the latent variable model, where $X_{\ell}$ is unobserved, consists of the marginal distributions of $X_V$ given by 
\[
p_{i_V} = \sum_{i_{\ell} = 1}^{r} p_{i_V, i_{\ell}} = \sum_{i_{\ell} = 1}^r \prod_{C \in \mathcal{C}(G)} \theta_{i_{C}, i_{\ell}}^{(C)}.
\]
Comparing with the parametrization of the mixture model $\text{Mixt}^r(\mathcal{M}_G)$ in \eqref{eq:parametrization}, we observe that the parametrizations, and hence the models, are equal.
\end{proof}

\begin{figure}[t]
\centering
(a)
\begin{tikzpicture}[align=center]
    \node[] (1) at (-1.1,0) {$1$};
    \node[] (2) at (0,0) {$2$};
    \node[] (3) at (1.1,0) {$3$};
    
    \draw[MyBlue] (1) edge (2);
    \draw[MyBlue] (2) edge (3);
\end{tikzpicture}
\hspace{2cm}
(b)
\begin{tikzpicture}[align=center]
    \node[] (1) at (-1.1,0) {$1$};
    \node[] (2) at (0,0) {$2$};
    \node[] (3) at (1.1,0) {$3$};
    \node[fill=lightgray] (4) at (0,1.1) {$4$};
    
    \draw[MyBlue] (1) edge (2);
    \draw[MyBlue] (2) edge (3);
    \draw[MyRed] (4) edge (1);
    \draw[MyRed] (4) edge (2);
    \draw[MyRed] (4) edge (3);
\end{tikzpicture}
\caption{The $3$-path and the corresponding graph with a latent node (gray).}
\label{fig:3-path}
\end{figure}
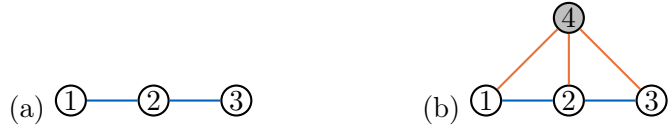

\begin{example}
Consider the $r$-mixture of the graphical model given by the graph in Figure \ref{fig:3-path} (a). The mixture model is equal to the latent variable model associated to the graph in Figure \ref{fig:3-path} (b) with latent node $\ell=4$, where the latent variable $X_4$ takes values in $[r]$.
\end{example}

On the algebra side, mixture models correspond to secant varieties. The \textit{$r$-th secant variety} of some variety $V$, denoted by $\text{Sec}^r(V)$, is the Zariski closure of the union of all affine linear planes spanned by any $r$ points in $V$; that is,
\[
    \text{Sec}^r(V):=\overline{\{\alpha_1 v^1 + \cdots + \alpha_r v^r : \alpha_1 + \cdots + \alpha_r=1 \text{ and } v^1, \ldots, v^r \in V\}}.
\]
Note that the Zariski closure of the mixture model equals the secant of the Zariski closure of the original model, that is, $\overline{\text{Mixt}^r(\mathcal{M})} = \text{Sec}^r(\overline{\mathcal{M}})$. We study vanishing ideals, dimensions, and maximum likelihood degrees of $\text{Mixt}^r(\M_G)$ for decomposable graphical models $\M_G$. 

\begin{definition}\label{def:expected-dim-defective}
The \textit{expected dimension} of the mixture model $\text{Mixt}^r(\M_G)$ is $r\cdot\dim(\M_G)+r-1$. If the dimension of a mixture model is less than the expected dimension, it is said to be \textit{defective}. 
\end{definition}

The following theorem will be used in Section \ref{sec:dimension} to compute dimensions of 2-mixtures.

\begin{theorem}\cite[Corollary 2.3]{draisma}\label{thm:Draisma}
Let $V_A$ be the toric variety specified by integer matrix $A \in \Z^{d\times n}$. Let $\bfh \in (\R^d)^*$. Let $A_+$ be the matrix consisting of columns $\bfa$ of $A$ such that $\bfh \cdot \bfa > 0$. Similarly, $A_-$ consists of the columns of $A$ such that $\bfh \cdot \bfa < 0$. Then 
        \[\mathrm{dim} \big(\mathrm{Sec}^2(V_A)\big) \geq \mathrm{rank} (A_+) + \mathrm{rank} (A_-) - 1. \]
    \end{theorem}

In particular, if we can separate the vertices of $\text{conv}(A)$ with a hyperplane so that the columns on either side both have full rank, then the secant has the expected dimension. Towards this end, we introduce the following definition.

\begin{definition} \label{def:SeparatingHyperplane}
    Let $A$ be an integer matrix with $d$ rows. Let $\bfh \in \big(\R^d\big)^\ast$. Adopting the notation in the statement of Theorem~\ref{thm:Draisma}, we say that $\bfh$ is a \emph{separating hyperplane} for $A$ if 
    \[
    \rank(A) = \rank(A_+) = \rank(A_-).
    \]
\end{definition}

\begin{example}
Consider the graphical model given by the graph $P_3 \sqcup P_1$ in Figure \ref{fig:running-example} (c), where each of the four random variables are binary. The hyperplane 
   \[
    \bfh = (1, 1, -1, 1 \ | \  -1, 1, -1, -1 \ | \ -1, 1)
    \]
is a separating hyperplane for the $A$-matrix of $\M_{P_3 \sqcup P_1}$, since in this case $\rk(A)=\rk(A_+)=\rk(A_-)=7$. On the other hand, the hyperplane
   \[
    \bfh = (1, -1, -1, 1 \ | \  -1, 1, -1, -1 \ | \ -1, 1).
    \]
is not separating, since $\rk(A_+)=6<\rk(A)$.
\end{example}

\section{Ideals of Secant Varieties} \label{sec:clique stars}

When we consider the $r$-th mixture of the model $\M_G\subseteq\Delta_n$, we are interested in the set of all polynomials that vanish on the distributions in the model. This mixture model $\Mixt^r(\M_G)$ lives inside the probability simplex $\Delta_{n}$, so it is a semialgebraic set. Taking the Zariski closure, we obtain the secant variety $\Sec^r(\overline{\M_G})$. Hence, the set of all polynomails that vanish on the distributions in $\Mixt^r(\M_G)$ is precisely the vanishing ideal $I_G^{(r)}$ of the secant variety, and the additional equation $p_1+p_2+\cdots+p_{n+1} = 1$ defining the simplex. Since we are interested in the equations of the prime ideal defining our model in complex space $\mathbb{C}^{n+1}$, we may relax the extra assumption that the coordinates have to sum to one, and only consider the ideal of the secant variety. This means that the dimension of the mixture model inside the probability simplex is the dimension of $I_G^{(r)}$ minus~one.

The next result generalizes Lemma~\ref{lem:2x2-minors} to the setting of secant varieties with $r\geq 2$. Let $I_{j_S; A\indep B}^{(r)}$ denote the ideal of all $(r+1)\times(r+1)$ minors of the matrix $M^{(r)}_{j_S; A\independent B}$, whose entries are the probabilities $p_{i_Ai_Bj_S+}$ with respect to $\text{Mixt}^r(\M_G)$ for every fixed value $j_S\in\mathcal{R}_S$. \looseness=-1

\begin{proposition}\label{prop:markov-ideal-containment}
The ideal $I_G^{(r)}$ contains 
\[
    \sum_{A \perp_{S} B \in G} \sum_{j_S\in \mathcal{R}_S} I_{j_S; A\indep B}^{(r)}.
\]
\end{proposition}
\begin{proof}
Fix $A\perp_S B\in G$ and $j_S\in \mathcal{R}_S$.  We claim that all minors in $I_{j_S; A\indep B}^{(r)}$ vanish on the parametrization of $\text{Sec}^r(\overline{\M_G})$. Each of the entries $p_{i_Ai_Bj_S+}$ is parametrized as
$$p_{i_Ai_Bj_S+}=\sum_{j_{V\setminus(A\cup B\cup S)}\in \mathcal{R}_{V\setminus (A \cup B \cup S)}}\sum_{\ell=1}^r\left(\prod_{C\in\mathcal{C}(G)}\theta^{(\ell,C)}_{i_{A\cap C}i_{B\cap C}j_{S\cap C}j_{(V\setminus (A\cup B\cup S))\cap C}}\right).$$
Exchanging the order of two summations and applying Lemma $\ref{lem:2x2-minors}$, we find that the matrix $M^{(r)}_{j_S; A\indep B}$ is a sum of $r$ matrices $M_{j_S; A\indep B}$ of rank one. Hence, $M^{(r)}_{j_S; A\indep B}$ has rank at most $r$. We conclude that all $(r+1)\times (r+1)$ minors vanish on the parametrization, as desired.
\end{proof}
\subsection{Clique Stars}
In this section, we focus on a special family of decomposable graphs, which we call clique stars. 
\begin{definition} \label{def:CliqueStar}
A graph $G$ is a \emph{clique star} if it is a union of maximal cliques, $G = \cup_{i=1}^k \widetilde{C}_i,$ and there is a clique $S$ such that $\widetilde{C}_i \cap \widetilde{C}_j = S$ for all $i \neq j$. In this case, we write $C_i = \widetilde{C}_i \setminus S$ and the vertex set of $G$ is the disjoint union, $C_1 \cup \dots \cup C_k \cup S$.
\end{definition}
The graphs below are examples of clique stars. 
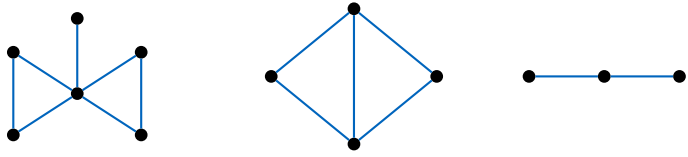
\begin{center}
\tikzset{
      every node/.style={circle, inner sep=0.1mm, minimum size=4pt, draw, thick, black, fill=black, text=black},
      every path/.style={thick}
}
\begin{minipage}{.2\textwidth}
\begin{tikzpicture}[align=center]
    \node[] (1) at (0,0) {};
    \node[] (2) at (0,1.0) {};
    \node[] (3) at (-0.85,0.55) {};
    \node[] (4) at (-0.85,-0.55) {};
    \node[] (5) at (0.85,0.55) {};
    \node[] (6) at (0.85,-0.55) {};
    
    \draw[MyBlue] (1) edge (2);
    \draw[MyBlue] (1) edge (3);
    \draw[MyBlue] (1) edge (4);
    \draw[MyBlue] (1) edge (5);
    \draw[MyBlue] (1) edge (6);
    \draw[MyBlue] (3) edge (4);
    \draw[MyBlue] (5) edge (6);
\end{tikzpicture}
\end{minipage}
\begin{minipage}{.2\textwidth}
\begin{tikzpicture}[align=center]
    \node[] (1) at (-1.1,0) {};
    \node[] (2) at (0,0.9) {};
    \node[] (3) at (0,-0.9) {};
    \node[] (4) at (1.1,0) {};
    
    \draw[MyBlue] (1) edge (2);
    \draw[MyBlue] (1) edge (3);
    \draw[MyBlue] (2) edge (4);
    \draw[MyBlue] (3) edge (4);
    \draw[MyBlue] (2) edge (3);
\end{tikzpicture}
\end{minipage}
\begin{minipage}{.2\textwidth}
\begin{tikzpicture}[align=center]
    \node[] (1) at (-1,0) {};
    \node[] (2) at (0,0) {};
    \node[] (3) at (1,0) {};
    
    \draw[MyBlue] (1) edge (2);
    \draw[MyBlue] (2) edge (3);
\end{tikzpicture}
\end{minipage}
\end{center}
Note that all graphs with two maximal cliques are clique stars. Moreover, graphs with an empty set of edges corresponding to the independence model are also clique stars, since we do not exclude $S=\emptyset$ in Definition \ref{def:CliqueStar}. For any clique star $G$, the following theorem expresses the dimension of $\text{Mixt}^r(\mathcal{M}_G)$ in terms of dimensions of the sets of certain rank $r$ tensors.

\begin{theorem}\label{thm:dim-clique stars}
    Let $G=(C_1 \cup \cdots \cup C_k \cup S, E)$ be a clique star. Then
    \[
    \dim\big(\mathrm{Mixt}^r(\mathcal{M}_G)\big) = \min\left\{d_S  \cdot \dim\Big(\mathcal{T}^{(r)}_{d_{C_1} \times \cdots \times d_{C_k}}\Big) - 1, \prod_{v \in V} d_v - 1 
    \right\},
    \]
    where $\mathcal{T}^{(r)}_{d_{C_1} \times \cdots \times d_{C_k}}$ is the set of $d_{C_1} \times \cdots \times d_{C_k}$ tensors of rank at most $r$.
\end{theorem}
\begin{proof}
If the secant variety $\text{Sec}^r(\overline{\mathcal{M}_G})$ fills the affine hull of the probability simplex then its dimension is exactly $\prod_{v \in V} d_v - 1$. Now assume that it does not fill the ambient space. Consider the Jacobian matrix $J$ of its parametrization, with columns labeled by the probabilities $p$ and rows labeled by the parameters $\theta$. Group the column labels into $d_S$ blocks: $\{(p_{i_{V\setminus S}j_S}: i_{V\setminus S}\in\mathcal{R}_{V\setminus S}): j_S\in \mathcal{R}_{S}\}$, one for each value of $j_S\in \mathcal{R}_{S}$. Group the row labels similarly: 
\[\Big\{\big(\theta^{(\ell,C)}_{i_{C \setminus S} j_S}: C\in\mathcal{C}(G), i_{C\setminus S}\in\mathcal{R}_{C\setminus S}, \ell\in[r]\big): j_S\in \mathcal{R}_{S}\Big\}.\]
With respect to this grouping, $J$ is block-diagonal. Each of the blocks is a copy of the Jacobian matrix of $\text{Sec}^r(\overline{\M_{C_1\indep\ldots\indep C_k}})$ where $\M_{C_1\indep\ldots\indep C_k}$ is the independence model on $k$ random variables, whose state space sizes are $d_{C_1},\ldots, d_{C_k}$, respectively. Hence, the rank of each block is $\dim(\mathcal{T}^{(r)}_{d_{C_1} \times \cdots \times d_{C_k}})$. Therefore, $\rank(J)=d_S\cdot \dim(\mathcal{T}^{(r)}_{d_{C_1} \times \cdots \times d_{C_k}})$ and since our mixture model lives in the simplex, we subtract one to obtain the desired dimension.
\end{proof}

\begin{corollary}
The mixture model $\text{Mixt}^r(\mathcal{M}_G)$ has the expected dimension if and only if the set of tensors of rank at most $r$ $\mathcal{T}^{(r)}_{d_{C_1} \times \cdots \times d_{C_k}}$ has the expected dimension as the $r$th secant of the variety of rank $1$ tensors. 
\end{corollary}

The dimension of the latter variety is known in many cases; see~\cite[Section 5.5]{landsberg2012tensors}. For matrices, we have the following corollary.

\begin{corollary}
Let $G=(C_1\cup C_2\cup S,E)$ be a clique star of two cliques. The mixture model $\text{Mixt}^r(\M_G)$ has dimension $\min\{d_S\cdot r(d_{C_1}+d_{C_2}-r)-1,d_{S}d_{C_1}d_{C_2}-1\}$. In particular, it is defective.
\end{corollary}
\begin{proof}
The dimension of the variety of all $d_{C_1}\times d_{C_2}$ matrices of rank at most $r$ is known to be $r(d_{C_1}+d_{C_2}-r)$~\cite[Proposition 1.1]{determinantal-rings}, so the conclusion follows from Theorem~\ref{thm:dim-clique stars}.
\end{proof}

\begin{example}[Mixture of paths] \label{ex:mixture-of-paths}
Consider the graphical model $\M_G$ given by the path $X_1 - X_2 - X_3$ where $d_1=5$ and $d_2=d_3=4$. This toric model has dimension $d_2(d_1+d_3-1)-1=31$ in $\Delta_{80}$. The ideal of the second secant variety $\text{Sec}^2(\M_G)$ is minimally generated by 160 cubics, such as
\begin{center}
\scalebox{0.9}{
$p_{344}p_{443}p_{542}-p_{343}p_{444}p_{542}-p_{344}p_{442}p_{543}+p_{342}p_{444}p_{543}+p_{343}p_{442}p_{544}-p_{342}p_{443}p_{544}$}.
\end{center}
All of these cubics are $3\times 3$ minors of the matrix of joint probabilities of $X_1$ and $X_3$, after we have fixed some value of $X_2$. The model $\text{Sec}^2(\M_G)$ has dimension $55=4\cdot 2(5+4-2)-1$, as opposed to the expected dimension $63=2\cdot 31+1$. Therefore, it is defective. The third secant variety has dimension 71 and is also defective, since it does not fill the ambient space. 
\end{example}

Next, we describe the ideal $I_G^{(r)}$ of the secant variety $\Sec^r(\overline{\M_G})$, where $G$ is a clique star. For any value $j_S\in\mathcal{R}_S$, let $\M_{j_S; C_1\indep\ldots\indep C_k}$ denote the independence model on the cliques upon fixing~$X_S=j_S$. Let $I^{(r)}_{j_S, d_{C_1} \times \cdots \times d_{C_k}}$ denote the vanishing ideal of $\text{Sec}^r(\overline{\M_{j_S; C_1\indep\ldots\indep C_k})}$.

\begin{theorem}\label{thm:clique star-ideals}
Let $G=(C_1 \cup \cdots \cup C_k \cup S, E)$ be a clique star. Then
\[
I_G^{(r)} = \sum_{j_S \in \mathcal{R}_S}   I^{(r)}_{j_S, d_{C_1} \times \cdots \times d_{C_k}}.
\]
\end{theorem}
\begin{proof}
Denote the right-hand side ideal by $J$. Note that for any fixed value of $j_S\in \mathcal{R}_S$, the polynomials in $ I^{(r)}_{j_S, d_{C_1} \times \cdots \times d_{C_k}}$ vanish on the parametrization of $\Sec^r(\overline{\M_G})$, as it is consistent with the parametrization of the independence model. Therefore, $J\subseteq I_G^{(r)}$. Since secant varieties are parametrized, they are irreducible. Hence, each ideal $I^{(r)}_{j_S, d_{C_1} \times \cdots \times d_{C_k}}$ is prime, which means that $J$ is also prime, as the sets of variables among the summands are disjoint. In addition, each ideal $I^{(r)}_{j_S, d_{C_1} \times \cdots \times d_{C_k}}$ has dimension $\dim(\mathcal{T}^r_{d_{C_1} \times \cdots \times d_{C_k}})$, so $J$ has dimension $d_S\dim(\mathcal{T}^r_{d_{C_1} \times \cdots \times d_{C_k}})$, since the variables are disjoint. By Theorem~\ref{thm:dim-clique stars}, both $I_G^{(r)}$ and $J$ have the same dimension. Since both ideals are prime and have the same dimension, and one is contained in another, they are indeed equal.
\end{proof}

When $r=2$, we can describe the ideal $I_G^{(2)}$ more concretely. It is generated in degree three by $3\times 3$ minors of tensor flattenings.

\begin{corollary}
For $r=2$, we have
$
I_G^{(2)} = \sum_{j_S\in \mathcal{R}_S} \sum_{A \perp_{S} B \in G} I_{j_S; A\indep B}^{(2)}.
$
In particular, the ideal is generated by $3 \times 3$ minors.
\end{corollary}
\begin{proof}
    The ideal $I^{(2)}_{j_S, d_{C_1} \times \cdots \times d_{C_k}}$ is generated by the $3\times 3$ minors of all flattenings of the corresponding $d_{C_1}\times \ldots \times d_{C_k}$ tensor $T_{j_S}$ of joint probabilities for every fixed $j_S\in\mathcal{R}_S$~\cite{allman2015tensors}. For a clique star, each saturated separation statement is of the form $A\perp_S B\in G$, where $A$ and $B$ are disjoint and $A\cup B=\cup_{i=1}^kC_i$. Thus, there is a one-to-one correspondence between flattenings of $T_{j_S}$ and saturated separation statements in $G$. Since any unsaturated separation statement in a graph is implied by a saturated separation statement, we have that $I^{(2)}_{j_S, d_{C_1} \times \cdots \times d_{C_k}}=\sum_{A \perp_{S} B \in G} I_{j_S; A\indep B}^{(2)}$, and the claim follows from Theorem~\ref{thm:clique star-ideals}.
\end{proof}

\subsection{Secants of Paths}
Mixtures of graphical models that are not given by clique stars are significantly more complicated. Their ideals are no longer generated by cubics, as illustrated in the next examples.
\begin{example}[A 4-path] \label{ex:4-path}
Let $G$ be the path $X_1-X_2-X_3-X_4$ on four nodes. Suppose $X_2$ is a ternary random variable, while the other three are binary. The mixture model $\Mixt^2(\M_G)$ has codimension 2 in $\Delta_{23}$. The corresponding secant variety is minimally generated by four sextics supported on 48 terms, such as
\begin{center}
\noindent\scalebox{0.75}{$-p_{1121}p_{1222}p_{1321}p_{2112}p_{2211}p_{2311}
+p_{1121}p_{1221}p_{1322}p_{2112}p_{2211}p_{2311}
+p_{1112}p_{1222}p_{1321}p_{2121}p_{2211}p_{2311}
-p_{1112}p_{1221}p_{1322}p_{2121}p_{2211}p_{2311}
$}\\ \scalebox{0.75}{$+p_{1122}p_{1221}p_{1321}p_{2111}p_{2212}p_{2311}
-p_{1121}p_{1221}p_{1322}p_{2111}p_{2212}p_{2311}
+p_{1111}p_{1221}p_{1322}p_{2121}p_{2212}p_{2311}
-p_{1111}p_{1221}p_{1321}p_{2122}p_{2212}p_{2311}
$}\\ \scalebox{0.75}{$-p_{1122}p_{1212}p_{1321}p_{2111}p_{2221}p_{2311}
+p_{1121}p_{1212}p_{1322}p_{2111}p_{2221}p_{2311}
-p_{1121}p_{1211}p_{1322}p_{2112}p_{2221}p_{2311}
+p_{1112}p_{1211}p_{1322}p_{2121}p_{2221}p_{2311}
$}\\ \scalebox{0.75}{$-p_{1111}p_{1212}p_{1322}p_{2121}p_{2221}p_{2311}
+p_{1111}p_{1212}p_{1321}p_{2122}p_{2221}p_{2311}
+p_{1121}p_{1211}p_{1321}p_{2112}p_{2222}p_{2311}
-p_{1112}p_{1211}p_{1321}p_{2121}p_{2222}p_{2311}
$}\\ \scalebox{0.75}{$-p_{1122}p_{1221}p_{1321}p_{2111}p_{2211}p_{2312}
+p_{1121}p_{1222}p_{1321}p_{2111}p_{2211}p_{2312}
-p_{1111}p_{1222}p_{1321}p_{2121}p_{2211}p_{2312}
+p_{1111}p_{1221}p_{1321}p_{2122}p_{2211}p_{2312}
$}\\ \scalebox{0.75}{$+p_{1122}p_{1211}p_{1321}p_{2111}p_{2221}p_{2312}
-p_{1111}p_{1211}p_{1321}p_{2122}p_{2221}p_{2312}
-p_{1121}p_{1211}p_{1321}p_{2111}p_{2222}p_{2312}
+p_{1111}p_{1211}p_{1321}p_{2121}p_{2222}p_{2312}
$}\\ \scalebox{0.75}{$+p_{1122}p_{1221}p_{1312}p_{2111}p_{2211}p_{2321}
-p_{1121}p_{1222}p_{1312}p_{2111}p_{2211}p_{2321}
+p_{1121}p_{1222}p_{1311}p_{2112}p_{2211}p_{2321}
-p_{1112}p_{1222}p_{1311}p_{2121}p_{2211}p_{2321}
$}\\ \scalebox{0.75}{$+p_{1111}p_{1222}p_{1312}p_{2121}p_{2211}p_{2321}
-p_{1111}p_{1221}p_{1312}p_{2122}p_{2211}p_{2321}
-p_{1122}p_{1221}p_{1311}p_{2111}p_{2212}p_{2321}
+p_{1111}p_{1221}p_{1311}p_{2122}p_{2212}p_{2321}
$}\\ \scalebox{0.75}{$+p_{1122}p_{1212}p_{1311}p_{2111}p_{2221}p_{2321}
-p_{1122}p_{1211}p_{1312}p_{2111}p_{2221}p_{2321}
-p_{1111}p_{1212}p_{1311}p_{2122}p_{2221}p_{2321}
+p_{1111}p_{1211}p_{1312}p_{2122}p_{2221}p_{2321}
$}\\ \scalebox{0.75}{$+p_{1121}p_{1211}p_{1312}p_{2111}p_{2222}p_{2321}
-p_{1121}p_{1211}p_{1311}p_{2112}p_{2222}p_{2321}
+p_{1112}p_{1211}p_{1311}p_{2121}p_{2222}p_{2321}
-p_{1111}p_{1211}p_{1312}p_{2121}p_{2222}p_{2321}
$}\\ \scalebox{0.75}{$-p_{1121}p_{1221}p_{1311}p_{2112}p_{2211}p_{2322}
+p_{1112}p_{1221}p_{1311}p_{2121}p_{2211}p_{2322}
+p_{1121}p_{1221}p_{1311}p_{2111}p_{2212}p_{2322}
-p_{1111}p_{1221}p_{1311}p_{2121}p_{2212}p_{2322}
$}\\ \scalebox{0.75}{$-p_{1121}p_{1212}p_{1311}p_{2111}p_{2221}p_{2322}
+p_{1121}p_{1211}p_{1311}p_{2112}p_{2221}p_{2322}
-p_{1112}p_{1211}p_{1311}p_{2121}p_{2221}p_{2322}
+p_{1111}p_{1212}p_{1311}p_{2121}p_{2221}p_{2322}$}.
\end{center}

These sextic generators were computed using the Multigraded Implicitization~\cite{joe-ben} package implemented in Macaulay2, which takes advantage of multigradings of polynomial maps to apply linear algebra techniques. 
\end{example}

\begin{example}[Binary path] \label{ex:binary-path}
Let $G$ be the path on five nodes where all random variables are binary. This is the smallest binary path $G$ such that $\Mixt^2(\M_G)$ does not fill the ambient space. The mixture model has the expected dimension 19 in $\Delta_{31}$. However, the generators of the ideal $I_G^{(2)}$ are quite complicated. By Proposition~\ref{prop:markov-ideal-containment}, the ideal $I^{(2)}_G$ contains 32 minimal cubic generators. These are $3$-minors of two $4\times 4$ matrices, obtained by fixing a value of $X_3$. In addition, the ideal has 57 minimal quartic generators, such as
\begin{center}
\noindent\scalebox{0.75}
{  $p_{11222}p_{21112}p_{22121}p_{22211}-p_{11112}p_{21222}p_{22121}p_{22211}-p_{11221}p_{21112}p_{22122}p_{22211}+
p_{11112}p_{21221}p_{22122}p_{22211}-$}
\scalebox{0.75}
{  $p_{11222}p_{21111}p_{22121}p_{22212}+
p_{11111}p_{21222}p_{22121}p_{22212}+p_{11221}p_{21111}p_{22122}p_{22212}-
p_{11111}p_{21221}p_{22122}p_{22212}-$}
\scalebox{0.75}
{  $p_{11212}p_{21122}p_{22111}p_{22221}+
p_{11122}p_{21212}p_{22111}p_{22221}+p_{11211}p_{21122}p_{22112}p_{22221}-
p_{11122}p_{21211}p_{22112}p_{22221}+$}
\scalebox{0.75}
{$p_{11212}p_{21121}p_{22111}p_{22222}-
p_{11121}p_{21212}p_{22111}p_{22222}-p_{11211}p_{21121}p_{22112}p_{22222}+
p_{11121}p_{21211}p_{22112}p_{22222}$}.
\end{center}
These quartics were computed using~\cite{joe-ben}.
Computation shows that the ideal generated by cubics and quadrics has codimension 12, so we have found all minimal generators of $\Mixt^2(\M_G)$.

\end{example}
\section{Dimensions of Second Secants} \label{sec:dimension}

In the last section, we fully characterized the dimension of mixtures of clique stars. In particular, mixtures of clique stars may be defective if the corresponding set of tensors is defective. Surprisingly, if we add cliques to the graph in such a way that it is no longer a clique star, we observe that the mixture model has expected dimension.

\begin{example} \label{ex:surprising-example}
    Consider the graph $P_3$ given by the path $X_1 - X_2 - X_3$ on three nodes pictured in Figure \ref{fig:running-example} (a). If the states are given by $d_1=5$ and $d_2=d_3=4$, then we have seen in Example~\ref{ex:mixture-of-paths} that the $2$-mixture model is defective. However, if we extend the path by one more node as in Figure \ref{fig:running-example} (b), we obtain that the 2-mixture has expected dimension for all possible number of states $d_4$. For example, if $d_4=2$, the toric model $\mathcal{M}_{P_4}$ has dimension $35$ in $\Delta_{160}$ and the two-mixture is of expected dimension $2 \cdot 35 + 1 = 71$. Similarly,  the $2$-mixtures of the graph obtained by adding an isolated node as in  Figure \ref{fig:running-example} (c) are also of expected dimension.
\end{example}

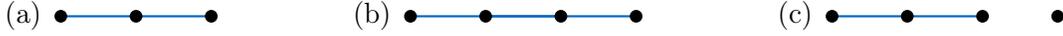
\begin{figure} 
\begin{center}
\tikzset{
      every node/.style={},
}
(a)
\begin{tikzpicture}[align=center]
    \draw[MyBlue] (-1,0)  -- (0,0);
    \draw[MyBlue] (0,0)  -- (1,0);
    \filldraw (-1,0) circle (2pt) node[anchor=south] {};
    \filldraw (0,0) circle (2pt) node[anchor=south] {};
    \filldraw (1,0) circle (2pt) node[anchor=south] {};
\end{tikzpicture}
\hspace{1.5cm}
(b)
\begin{tikzpicture}[align=center]
    \draw[MyBlue] (-1,0)  -- (0,0);
    \draw[MyBlue] (0,0)  -- (1,0);
    \draw[MyBlue] (0,0)  -- (2,0);
    \filldraw (-1,0) circle (2pt) node[anchor=south] {};
    \filldraw (0,0) circle (2pt) node[anchor=south] {};
    \filldraw (1,0) circle (2pt) node[anchor=south] {};
    \filldraw (2,0) circle (2pt) node[anchor=south] {};
\end{tikzpicture}
\hspace{1.5cm}
(c)
\begin{tikzpicture}[align=center]
    \draw[MyBlue] (-1,0)  -- (0,0);
    \draw[MyBlue] (0,0)  -- (1,0);
     \filldraw (-1,0) circle (2pt) node[anchor=south] {};
    \filldraw (0,0) circle (2pt) node[anchor=south] {};
    \filldraw (1,0) circle (2pt) node[anchor=south] {};
    \filldraw (2,0) circle (2pt) node[anchor=south] {};
\end{tikzpicture}
\end{center}
\caption{The paths $P_3$, $P_4$ and $P_3 \sqcup P_1$.} 
\label{fig:running-example}
\end{figure}

The goal of this section is to prove Theorem~\ref{thm:ExpectedDimension}, which states that the second secant of a decomposable graphical model that is not a clique star has the expected dimension. To accomplish this, we prove several lemmas that allow us to construct separating hyperplanes; recall Definition~\ref{def:SeparatingHyperplane}. In each of these lemmas, we begin with a graph $G$ and an assignment of states $\bfd$ such that the $A$-matrix $A_{G,\bfd}$ has a separating hyperplane. Then we perform an operation to this graph and show that the hyperplane can be extended to a separating hyperplane for the $A$-matrix of the resulting graph. This allows us to apply Theorem~\ref{thm:Draisma}, which relates these separating hyperplanes to the dimension of the secant variety. We make use of the following linear algebra lemma regarding ranks of block matrices. \looseness=-1

\begin{lemma}\label{lem:RankOfStackedMatrix}
Let $M_1$ and $M_2$ be matrices with the same number of columns $n$. For each $i$, let $M_i'$ be a matrix consisting of a subset of $n' \leq n$ columns of $M_i'$. If
    \[
    \rank \begin{pmatrix}
        M_1 \\
        M_2
    \end{pmatrix} = \rank \begin{pmatrix}
        M_1' \\
        M_2'
    \end{pmatrix},
    \]
    then $\rank(M_i) = \rank(M_i')$ for each $i = 1,2$.
\end{lemma}

\begin{proof}
    Let $M_1 \in \mathbb{R}^{k_1 \times n}$, $M_2 \in \mathbb{R}^{k_2 \times n}$, and let 
    \[
    M = \begin{pmatrix}
        M_1 \\
        M_2
    \end{pmatrix} \qquad \text{and} \qquad M' = \begin{pmatrix}
        M_1' \\
        M_2'
    \end{pmatrix}.
    \]
   First, observe that $\rank(M_i') \leq \rank(M_i)$ for $i=1,2$ since $M_i'$ consists of a subset of the columns of $M_i$.
    Now, assume that $\rank(M_1') < \rank(M_1) $. Then the image of $M_1'$ is contained in a strictly lower dimensional subspace of the image of $M_1$.  In particular, there is a vector $\bfx_1 \in \mathbb{R}^{k_1}$ that is in the image of $M_1$ but not in the image of $M_1'$. Since $\bfx_1$ is in the image of $M_1$, there is $\bfy \in \mathbb{R}^{n}$ such that $M_1 \cdot \bfy = \bfx_1$. Now, define $\bfx_2 = M_2 \cdot \bfy \in \mathbb{R}^{k_2}$, and observe that the vector $(\bfx_1^{\top}, \bfx_2^{\top})^{\top} $ is in the image of $M$ but not in the image of $M'$. Hence, the image of $M'$ must have strictly lower dimension than the image of $M$, that is, $\rank(M') < \rank(M)$. Since this is a contradiction, we conclude that $\rank(M_1') = \rank(M_1)$. Equivalently, we can see that $\rank(M_2') = \rank(M_2)$.
\end{proof}

We now consider the operation of adding a binary random variable to $G$ while preserving its maximal cliques and hence, the block structure of the $A$-matrix of the graphical model.
Let $G$ be a graph with maximal cliques $C_1,\dots,C_k$. 
We wish to add a node to the intersection of some of these maximal cliques in a way that preserves the clique structure of $G$. 

\begin{definition}
    Let $I \subseteq [k]$ be a set of indices of the maximal cliques of $G$. Let $x$ be a node not in $G$. The graph $G +_I x$ is is the graph obtained from $G$ by adding the node $x$ so that $x$ is adjacent to all nodes of each $C_i$ for $i \in I$. The set $I$ is \emph{compatible} with $G$ if the maximal cliques of $G +_I x$ are exactly $C_i \cup \{ x \}$ for $i \in I$ and $D_j$ for $j \not\in I$.
\end{definition}

\begin{example}
    Consider the path $P_4$ with four vertices $1,2,3$ and $4$ in that order. It has three maximal cliques $C_1 = \{1,2\}$, $C_2 = \{2,3\}$ and $C_3 = \{3,4\}$. Then $I = \{1,2\}$ is compatible since adding a vertex $x$ to cliques $C_1$ and $C_2$ yields a graph whose maximal cliques are $\{1,2,x\}$, $\{2,3,x\}$ and $\{3,4\}$. However $I' = \{1,3\}$ is not compatible since adding $x$ to both $C_1$ and $C_3$ also implies that $x$ is added to $C_2.$
\end{example}

Suppose that $G$ has maximal cliques $C_1,\dots,C_k$ and that $I$ is compatible with $G$. Without loss of generality, we assume that $I = [\ell]$ for $\ell \leq k$.
The columns of the matrix $A_{G,\bfd}$ are indexed by $\bfi \in [d_1] \times \dots \times [d_n]$ where $n$ is the number of vertices of $G$. Its rows are organized into blocks corresponding to each clique $C_j$ and indexed by $\bfi_{C_j} \in \prod_{v \in C_j} [d_v]$. Adding a binary variable corresponding to vertex $x$ results in the matrix $A_{G+_Ix,(\bfd,2)}$ whose columns are indexed by $\bfi 1$ and $\bfi 2$ where $\bfi \in [d_1] \times \dots \times [d_n]$. Since $I$ is compatible, the maximal cliques of $G +_I x$ are $C_j \cup \{x \}$ when $j \leq \ell$ and $C_j$ when $j > \ell$. So the rows of $A_{G+_Ix,(\bfd,2)}$ are also organized into blocks corresponding to each of these $k$ cliques. If $j \leq \ell$, the rows in this block are indexed by $\bfi_{C_j} 1$ and $\bfi_{C_j}2$ where $\bfi_{C_j} \in \prod_{v \in C_j} [d_v]$. If $j > \ell$, the rows are indexed in the same way as in $A_{G,\bfd}$. So the matrix $A_{G+_Ix,(\bfd,2)}$ is of the form

\begin{equation*}
    A_{G+_I x,(\bfd,2)} = \begin{blockarray}{ccc}
    & \bfi 1 & \bfi 2 \\
    \begin{block}{c(cc)}
       \bfi_{C_1} 1 & A_1 & \mathbf{0} \\
       \bfi_{C_1} 2 & \mathbf{0} & A_1 \\
       \BAhline
       & \vdots & \vdots \\
       \BAhline
       \bfi_{C_\ell}1 & A_\ell & \mathbf{0} \\
       \bfi_{C_\ell}2 & \mathbf{0} & A_\ell \\
       \BAhline
       \bfi_{C_{\ell+1}} & A_{\ell+1} & A_{\ell+1} \\
       \BAhline
       & \vdots & \vdots \\
       \BAhline
       \bfi_{C_k} & A_k & A_k \\
    \end{block}
    \end{blockarray}.   
\end{equation*}

In Theorem~\ref{thm:AddBinaryRV}, we will show that if $\bfh$ is a separating hyperplane for $A_{G,\bfd}$, then we can use it to construct a separating hyperplane for $A_{G+_I x, (\bfd, 2)}$ as long as $I$ is compatible with $G$. We construct such a hyperplane $\bfhbar$ by setting
\begin{align}\label{eq:AddBinaryRVHyperplane}
    \bar{h}_{j,\bfi} &  = h_{j,\bfi} & & \text{ for } j\not\in I \text{ and } \bfi \in \prod_{y \in C_j} [d_y], \text{ and } \nonumber \\
    \bar{h}_{j,(\bfi,i_x)}  & = h_{j,\bfi} & & \text{ for } j\in I, \ \bfi \in \prod_{y \in C_j} [d_y] \text{ and } i_x \in \{1,2\}.
\end{align}

\begin{example}\label{ex:AddNode}
Consider the graph $G = P_3 \sqcup P_1$ as in Figure \ref{fig:running-example} (c). It has three maximal cliques, $C_1 = \{1,2\}$, $C_2 = \{2,3\}$ and $C_3 = \{4\}$, where the nodes $1, \ldots, 4$ are numbered from left to right. A separating hyperplane for $A_{G, (2,2,2,2)}$ is
\[ \small{
\mathbf{h} = \begin{blockarray}{cccc|cccc|cc}
\begin{block}{\BAmulticolumn{4}{c} | \BAmulticolumn{4}{c} | \BAmulticolumn{2}{c}}
    C_1 & C_2 & C_3 \\
\end{block}
11\cdot\cdot & 12\cdot\cdot & 21\cdot\cdot & 22\cdot\cdot & \cdot 11\cdot & \cdot 12\cdot & \cdot 21\cdot & \cdot 22\cdot & \cdot\cdot\cdot 1 & \cdot\cdot\cdot 2 \\
\begin{block}{(cccc|cccc|cc)}
  1 & 1 & -1 & 1 & - 1 & 1 & -1 & -1 &  -1 & 1 \\
\end{block}\end{blockarray}.}
\]
Note that the set $I = [4]$ is compatible with $G$ since coning over a graph preserves its maximal clique structure. So we may add a node $5$ which is adjacent to all other nodes and whose associated random variable is binary. Since $5$ is added to every clique, we duplicate each block of $\bfh$ to obtain $\bfhbar$. The resulting hyperplane $\bfhbar$ is 
\[
(1,1,-1,1,1,1,-1,1 \ | \ -1,1,-1,-1,-1,1,-1,-1 \ | \ -1,1,-1,1),
\]
where the blocks correspond to $C_1 \cup \{5\}$, $C_2 \cup \{5\}$ and $C_3 \cup \{5\}$ respectively. The columns within the blocks are ordered lexicographically first according to the value assigned to the value at node $5$ and then in the order in which the other values appear in $\bfh$.
\end{example}

\begin{theorem}\label{thm:AddBinaryRV}
Let $G$ be a graph with maximal cliques $C_1,\dots,C_k$. Let $\bfd \in \N^{|V(G)|}$ be the list of the number of states of the random variable at each node. Let $I \subset [k]$ be compatible with $G$. Suppose that $\bfh$ is a separating hyperplane for $A_{G,\bfd}$ and let $\bfhbar$ be obtained from $\bfh$ as in Equation~\eqref{eq:AddBinaryRVHyperplane}. Then $\bfhbar$ is a separating hyperplane for $A_{G+_I x, (\bfd,2)}.$ 
\end{theorem}

\begin{proof}
    Without loss of generality, let $I = [\ell]$ for some $\ell \leq k$. 
    Let $P$ be the matrix consisting of the columns $\bfa$ of $A_{G,\bfd}$ such that $\bfh \cdot \bfa > 0$. Let $P_1,\dots,P_k$ be the blocks of $P$ corresponding to each clique $C_i$. Let $\bar{P}$ be the matrix consisting of columns $\bfa$ of  $A_{G +_I x, (\bfd,2)}$ such that $\bfhbar \cdot \bfa > 0$. Then $\overline P$ has the same block structure as $A_{G +_I x, (\bfd,2)}$; that is, it is of the form
    \[
    \overline P = \begin{blockarray}{ccc}
    & \bfi 1 & \bfi 2 \\
    \begin{block}{c(cc)}
       \bfi_{C_1} 1 & P_1 & \mathbf{0} \\
       \bfi_{C_1} 2 & \mathbf{0} & P_1 \\
       \BAhline
       & \vdots & \vdots \\
       \BAhline
       \bfi_{C_\ell}1 & P_\ell & \mathbf{0} \\
       \bfi_{C_\ell}2 & \mathbf{0} & P_\ell \\
       \BAhline
       \bfi_{C_{\ell+1}} & P_{\ell+1} & P_{\ell+1} \\
       \BAhline
       & \vdots & \vdots \\
       \BAhline
       \bfi_{C_k} & P_k & P_k \\
    \end{block}
    \end{blockarray} .   
    \]

Note that by Lemma~\ref{lem:RankOfStackedMatrix} and the fact that $\rank(P) = \rank(A_{G,\bfd})$ we have
\[
\rank \begin{pmatrix}
    P_1 \\
    \vdots \\
    P_\ell  
\end{pmatrix} = \rank \begin{pmatrix}
    A_1 \\
    \vdots \\
    A_\ell  
\end{pmatrix}.
\]
Hence, we have
\[
\rank (\overline{P}) = \rank (P) + \rank \begin{pmatrix}
    P_1 \\
    \vdots \\
    P_\ell  
\end{pmatrix} = \rank (A) + \rank \begin{pmatrix}
    A_1 \\
    \vdots \\
    A_\ell  
\end{pmatrix} = \rank(A_{G +_I x, (\bfd,2)}).
\]
The same holds for the matrix consisting of columns of $A_{G +_I x, (\bfd,2)}$ whose dot product with $\bfhbar$ is negative. So $\bfhbar$ is a separating hyperplane for $A_{G +_I x, (\bfd,2)}$.
\end{proof}

We have shown how to extend the hyperplane for a graph when we add a binary random variable in a way that preserves the graph's clique structure. Next, we must consider how to extend the hyperplane when we increase the number of states at a node of $G$. Let $\bfd =(d_1, \ldots, d_n) \in \N^{|V(G)|}$ be the list of the number of states of the random variables at the nodes of $G$ and let $\bfh$ be a separating hyperplane for $A_{G, \bfd}$. Without loss of generality, consider the list $\bfdbar := (d_1 +1 ,d_2,\dots,d_n)$, where $n=|V(G)|$ is the number of nodes in $G$. Let $G$ have maximal cliques $C_1,\dots,C_k$ and suppose without loss of generality that $v_1 \in C_1,\dots,C_{\ell}$ and $v_1 \not\in C_{\ell+1}, \dots, C_k$. We construct the hyperplane $\bfhbar$ for $A_{G,\bfdbar}$ by
\begin{align}\label{eq:IncreaseStateHyperplane}
    \bar{h}_{j,\bfi} &  = h_{j,\bfi} & & \text{ for all } j \in [k] \text{ if } i_1 \neq d_1 + 1, \text{ and } \nonumber \\
    \bar{h}_{j,(d_1+1,\bfi)}  & = h_{j,(d_1,\bfi)} & & \text{ for } j \in [\ell].
\end{align}

\begin{example}\label{ex:IncreaseStates}
    We again consider the graph $G = P_3 \cup P_1$ as in Figure \ref{fig:running-example} (c) where all random variables are binary. Consider increasing the number of states at node 1 by one. Using the same hyperplane $\bfh$ for $A_{G,(2,2,2,2)}$ from Example~\ref{ex:AddNode}, the resulting hyperplane for $A_{G,(3,2,2,2)}$  is
    \[
    \small{
 \begin{blockarray}{ccccccc|cccc|cc}
\begin{block}{c&\BAmulticolumn{6}{c} | \BAmulticolumn{4}{c} | \BAmulticolumn{2}{c}}
   & C_1 & C_2 & C_3 \\
\end{block}
& 11\cdot\cdot & 12\cdot\cdot & 21\cdot\cdot & 22\cdot\cdot & 31\cdot\cdot & 32\cdot\cdot & \cdot 11\cdot & \cdot 12\cdot & \cdot 21\cdot & \cdot 22\cdot & \cdot\cdot\cdot 1 & \cdot\cdot\cdot 2 \\
\begin{block}{c(cccccc|cccc|cc)}
  \bfhbar = & 1 & 1 & -1 & 1 & -1 & 1 & - 1 & 1 & -1 & -1 &  -1 & 1 \\
\end{block}\end{blockarray}.}
\]
\end{example}
\begin{theorem} \label{thm:IncreaseStates}
Let  $G$ be a graph and let $\bfd \in \N^{|V(G)|}$ be the list of the number of states of the associated random variables. Let $\bfh$ be a separating hyperplane for $A_{G,\bfd}$, and let $\bfdbar = (d_1 +1, d_2, \dots, d_n)$. Then the hyperplane $\bfhbar$ defined in Equation \eqref{eq:IncreaseStateHyperplane} is separating for $A_{G,\bfdbar}.$
\end{theorem}

\begin{proof}
   Reordering the rows of $A_{G,\bfd}$ does not change the graphical model, so we consider rearranging the rows according to the value of $i_1$. Since $v_1 \in C_1,\dots,C_{\ell}$, we put the row corresponding to $\theta^{(C_j)}_{\bfi}$ before the row corresponding to $\theta^{(C_{j'})}_{\bfi'}$ if $j, j' \in [\ell]$ and $i_1 < i_1'$. The rows corresponding to cliques $C_{\ell+1},\dots,C_k$ remain the same. We are especially interested in the rows of $A_{G,\bfd}$ and $A_{G,\bfd'}$ corresponding to $\theta^{(C_j)}_{\bfi}$ with $j \in [\ell]$ and $i_1 \geq d_1$. So with this ordering of the rows, these matrices have the form
    \begin{equation}\label{eq:AddStatesMatrixForm}
A_{G,\bfd} = \begin{blockarray}{ccc}
    \begin{block}{c(cc)}
       i_1 < d_1 & A & \mathbf{0} \\
       \BAhline
       i_1 = d_1 & \mathbf{0} & B \\
       \BAhline
       C_j: j > \ell & C & D \\
    \end{block}
    \end{blockarray} \qquad \text{ and } \qquad
    A_{G,\bfdbar} = \begin{blockarray}{cccc}
    \begin{block}{c(ccc)}
       i_1 < d_1 & A & \mathbf{0} & \mathbf{0} \\
       \BAhline
       i_1 = d_1 & \mathbf{0} & B & \mathbf{0} \\
       \BAhline
       i_1 = d_1 +1 & \mathbf{0} & \mathbf{0} & B \\
       \BAhline
       C_j: j > \ell & C & D & D\\
    \end{block}
    \end{blockarray}
    \end{equation}

     Without loss of generality, let $P$ be the matrix consisting of columns $\bfa$ of $A_{G,\bfd}$ such that $\bfh \cdot \bfa > 0$. Let $\overline{P}$ be the matrix consisting of columns $\bfabar$ of $A_{G,\bfdbar}$ such that $\bfhbar \cdot \bfabar > 0$. Then by our construction of $\bfhbar$, if $i_1 \neq d_1+1$, then the column $\bfabar_{\bfi}$ of $A_{G,\bfdbar}$ is a column of $\overline{P}$ if and only if $\bfa_{\bfi}$ of $A_{G,\bfd}$ is a column of $P$. The column $\bfabar_{d_1+1,\bfi}$  is a column of $\overline{P}$ if and only if $\bfa_{d_1,\bfi}$ is a column of $P$. Hence $P$ and $\overline{P}$ have the same block structure as $A_{G,\bfd}$ and $A_{G,\bfdbar}$ in Equation~\eqref{eq:AddStatesMatrixForm}. That is, if 
    \begin{equation*}
P= \begin{blockarray}{ccc}
    \begin{block}{c(cc)}
       i_1 < d_1 & Q & \mathbf{0} \\
       \BAhline
       i_1 = d_1 & \mathbf{0} & R \\
       \BAhline
       C_j: j > \ell & S & T \\
    \end{block}
    \end{blockarray} \qquad \text{ then } \qquad
    \overline{P} = \begin{blockarray}{cccc}
    \begin{block}{c(ccc)}
       i_1 < d_1 & Q & \mathbf{0} & \mathbf{0} \\
       \BAhline
       i_1 = d_1 & \mathbf{0} & R & \mathbf{0} \\
       \BAhline
       i_1 = d_1 +1 & \mathbf{0} & \mathbf{0} & R \\
       \BAhline
       C_j: j > \ell & S & T & T\\
    \end{block}
    \end{blockarray}
    \end{equation*}

Since $\rank(A_{G,\bfd}) = \rank(P)$, by Lemma~\ref{lem:RankOfStackedMatrix}, we have that $\rank(R) = \rank(B)$. Hence,
\[
\rank(\overline{P}) = \rank(P) + \rank(R) = \rank(A_{G,\bfd}) + \rank(B) = \rank(A_{G,\bfdbar}),
\]
as needed. So $\bfhbar$ is a separating hyperplane for $A_{G,\bfdbar}$.
\end{proof}

The following proposition shows that all graphs with three maximal cliques that are not clique stars can be obtained from two minimal graphs by adding nodes to compatible subsets of the maximal cliques. These graphs will serve as a basis for induction to obtain separating hyperplanes for all decomposable graphs that are not clique stars.

\begin{proposition}\label{prop:ThreeCliques}
    Let $G$ be a graph with three maximal cliques that is not a clique star. Then $G$ can be obtained from $P_4$ or $P_3 \sqcup P_1$ by repeatedly adding a node to a compatible subset of maximal cliques.
\end{proposition}

\begin{proof}
    We induct on the number of nodes. The graphs $P_4$ and $P_3 \sqcup P_1$, pictured in Figure \ref{fig:running-example}, are exactly the graphs on four nodes that are not clique stars and that have three maximal cliques. 
    
    Let $G$ be a graph on $n > 4$ nodes which is not a clique star and which has maximal cliques $C_1, C_2$ and $C_3$. If $C_1 \cap C_2 \cap C_3$ is nonempty, then we can remove one of its elements without changing the maximal clique structure of $G$. Indeed let $x \in C_1 \cap C_2 \cap C_3$ and  let $C$ be a  clique in $G - x$. Then $C$ is also a clique in $G$ and hence is contained in some $C_i$. Since $x$ does not belong to $C$, $C$ is contained in $C_i -x$. Moreover, each $C_i -x$ is maximal since none is contained in another. Hence the maximal cliques of $G -x$ are $C_i -x$ for $i = 1,2,3$. Thus $\{ C_1 - x, C_2 -x, C_3 -x \}$ is a compatible collection of cliques in $G-x$, and we may add $x$ to them to obtain $G$.

    Suppose that one of $C_i \setminus (C_j \cup C_k)$ or $(C_i \cap C_j) \setminus C_k$ for distinct $i,j,k \in [3]$ have two or more elements. We claim that removing one of these elements does not change the maximal clique structure of $G$. In the first case, suppose without loss of generality that $x$ and $y$ belong to $C_1$ but not $C_2$ or $C_3$. We claim that $C_1-y$, $C_2$ and $C_3$ are maximal cliques in $G - y$. Indeed, since $C_2$ and $C_3$ are unchanged by removing $y$, they remain maximal in $G - y$. Since $x \in (C_1 -y) \setminus (C_2 \cup C_3)$, we have that $C_1 - y$ is not contained in $C_2$ or $C_3$. So it is also maximal. If $C$ is a clique in $G-y$, then it is also a clique in $G$ and hence is contained in some $C_i$. Since $y \not\in C$, $C$ is also contained in $C_i - y$, as needed. Finally, $G-y$ is not a clique star since the pairwise intersections of maximal cliques are not changed by removing $y$. Hence, we can obtain $G$ from $G-y$ by adding $y$ to $C_1 - y$.

    In the second case, suppose that $x$ and $y$ belong to $C_1 \cap C_2$ and not to $C_3$. As in the previous case, $C_1 - y$, $C_2 - y$ and $C_3$ are maximal cliques in $G-y$. Indeed, $x \in C_1 -y, C_2 -y$, so neither is contained in $C_3$. And neither of $C_1 -y$ and $C_2 -y$ is contained in the other since neither of $C_1$ and $C_2$ are contained in the other and both contain $y$. Similarly, all other cliques in $G-y$ are contained in $C_1 -y$,$C_2 -y$ or $C_3$. Finally, $x \in (C_1 -y) \cap (C_2 - y)$ but $x \not\in (C_1 - y) \cap C_3$. So $G-y$ is not a clique star. Thus $\{C_1 -y, C_2 -y\}$ is a compatible collection of cliques in $G-y$ and we can obtain $G$ by adding $y$ to both of them.

    It remains to consider the non-clique stars on three cliques such that $C_1 \cap C_2 \cap C_3 = \emptyset$ and all of $C_i \setminus (C_j \cup C_k)$ and $(C_i \cap C_j) \setminus C_k$ have size at most one. Such a graph on more than four nodes must have five or six nodes. There are three cases.

    \emph{Case 1:} Suppose $G$ has six nodes so that all of $C_i \setminus (C_j \cup C_k)$ and $(C_i \cap C_j) \setminus C_k$ have size one. Then each $C_i$ is a triangle, their pairwise intersections each have one node, and the intersection of all three is empty. The only such graph is pictured in Figure \ref{fig:ThreeCliqueCase1}, which we can see has four maximal cliques. Hence, this case is impossible.

    \emph{Case 2:} Suppose all of $C_i \setminus (C_j \cup C_k)$ and $(C_i \cap C_j) \setminus C_k$ have size one except for $C_1 \setminus (C_2 \cup C_3) = \emptyset$. Let $C_1 \cap C_2 = \{x \}$, $C_1 \cap C_3 = \{y\}$ and $C_2 \cap C_3 = \{z \}$, as pictured in Figure \ref{fig:ThreeCliqueCase2}. But then $\{x,y,z\}$ is a clique that properly contains $C_1 = \{x,y\}$, which is a contradiction.

    \emph{Case 3:} Suppose all of $C_i \setminus (C_j \cup C_k)$ and $(C_i \cap C_j) \setminus C_k$ have size one except for $(C_1 \cap C_3) \setminus C_2= \emptyset$. This case is pictured in Figure \ref{fig:ThreeCliqueCase3} and can be obtained for $P_4$ by adding a node to the middle clique.

    We have exhausted all possible cases and shown by induction that all non-clique stars with three maximal can be obtained from $P_4$ or $P_3 \sqcup P_1$ by repeatedly adding nodes to a compatible set of maximal cliques.
\end{proof}

\begin{figure}
\centering
\begin{subfigure}{0.3 \textwidth}
\centering
    \begin{tikzpicture}
    \draw[MyBlue] (0,0) -- (2,0) -- (1,1.5) -- (0,0);
    \draw[MyBlue] (0.5,0.75) -- (1,0) -- (1.5,0.75) -- (0.5,0.75);
    \filldraw (0,0) circle (2pt);
    \filldraw (2,0) circle (2pt);
    \filldraw (1,0) circle (2pt);
    \filldraw (1,1.5) circle (2pt);
    \filldraw (0.5,0.75) circle (2pt);
    \filldraw (1.5,0.75) circle (2pt);
    \end{tikzpicture}
    \caption{Case 1.}
    \label{fig:ThreeCliqueCase1}
\end{subfigure}
\begin{subfigure}{0.3\textwidth}
\centering
 \tikzset{
     every node/.style={},
     every path/.style={thick}
}
    \begin{tikzpicture}
     \draw[MyBlue] (0,0) -- (2,0) -- (1.5,0.75) -- (0.5, 0.75) -- (0,0);
    \draw[MyBlue](1.5,0.75) -- (1,0) -- (0.5,0.75);
    \filldraw (0,0) circle (2pt);
    \filldraw (2,0) circle (2pt);
    \filldraw (1,0) circle (2pt) node[anchor = north] {$z$};
    \filldraw (0.5,0.75) circle (2pt) node[anchor = south] {$x$};
    \filldraw (1.5,0.75) circle (2pt) node[anchor = south] {$y$};
    \end{tikzpicture}
    \caption{Case 2.}
    \label{fig:ThreeCliqueCase2}
\end{subfigure}
\begin{subfigure}{0.3\textwidth}
\centering
    \begin{tikzpicture}
    \draw[MyBlue] (0,0) -- (3,0);
    \draw[MyBlue] (1,0) -- (1.5,0.75) -- (2,0);
    \filldraw (0,0) circle (2pt);
    \filldraw (2,0) circle (2pt);
    \filldraw (1,0) circle (2pt);
    \filldraw (3,0) circle (2pt);
    \filldraw (1.5,0.75) circle (2pt);
    \end{tikzpicture}
    \caption{Case 3.}
    \label{fig:ThreeCliqueCase3}
\end{subfigure}
\caption{Graphs with five and six nodes in the proof of Proposition~\ref{prop:ThreeCliques}}
\end{figure}
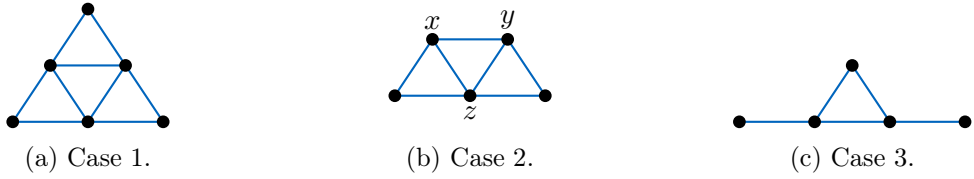

It follows from these results that the second secant of the graphical model on any non-clique star with three cliques has the expected dimension.

\begin{corollary}\label{cor:ThreeCliques}
    Let $G$ be a non-clique star with three maximal cliques and let $\bfd \in \N^{|V(G)|}$ be the list of the number of states of the random variables at the nodes of $G$, each entry of which is greater than or equal to $2$. Then there exists a separating hyperplane of $A_{G,\bfd}$. In particular, the second mixture of the corresponding graphical model has the expected dimension.
\end{corollary}

\begin{proof}
    First, we note that the $A$-matrices for the binary graphical models on $P_4$ and $P_3 \sqcup P_1$, pictured in Figures \ref{fig:running-example} (b) and (c), each have a separating hyperplane. A hyperplane that separates $A_{P_4, (2,2,2,2)}$ is 
    \[(1, -1, -1, 1 \ | \ 1, -1, -1, 1 \ | \ -1, 1, -1, 1)
    \]
    where the three cliques are ordered from left to right and the assignment of states are ordered reverse lexicographically within each section. Similarly, a hyperplane that separates $A_{P_3 \sqcup P_1, (2,2,2,2)}$~is 
    \[
    (1, 1, -1, 1 \ | \  -1, 1, -1, -1 \ | \ -1, 1).
    \]
    
    By Proposition~\ref{prop:ThreeCliques}, $G$ can be obtained from one of these minimal graphs by repeatedly adding a node to a compatible subset of the maximal cliques of $G$. Hence by Theorem~\ref{thm:AddBinaryRV}, there exists a separating hyperplane for the $A$-matrix of the binary graphical model on $G$. By repeatedly increasing the number of states at each node $i$ until it reaches $d_i$, Theorem~\ref{thm:IncreaseStates} shows that there exists a separating hyperplane for the $A$-matrix of the graphical model on $G$ with states $\bfd$. Hence by Theorem~\ref{thm:Draisma}, the second mixture of this model has the expected dimension.
\end{proof}

In order to prove that the second mixture of \emph{all} decomposable non-clique stars have the expected dimension, we make use of the clique-sum operation. This plays a critical role in the proof of the main result since all decomposable graphs can be obtained as clique-sums of cliques.

\begin{definition}\cite[Section 4.2]{minor-graph-theory}\label{def:CliqueGluing}
    Let  $G$ be an arbitrary graph and let $K$ be another graph that is complete. Let $S$ be a (potentially empty) clique in $G$. The graph $G +_S K$ is obtained from $G$ by adding $K$ and the edges $\{k,s\}$ for all $k \in V(K)$ and all $s \in S$. This graph is the \emph{clique sum of $G$ and $K$ along~$S$.}
\end{definition}

The next lemma is the main tool for extending a separating hyperplane when we add a new clique to the graph via a clique-sum.

\begin{lemma} \label{lem:StatesInS}
Let  $G$ be a graph and let $\bfd \in \N^{|V(G)|}$ be the list of the number of states of the associated random variables. Let $P$ be any matrix consisting of a subset of the columns of $A_{G,\bfd}$ such that $\rk(P)= \rk(A_{G,\bfd})$. Then, for any clique $S$  and any fixed state $j_S \in \mathcal{R}_S$, there is a column in $P$ that is indexed by 
$i \in \mathcal{R}$ such that $i_S = j_S$.
\end{lemma}
\begin{proof}
If $S = \emptyset$ this is trivial since $\rk(A_{G,\bfd}) > 0$. Otherwise, note that the matrix $A_{G,\bfd}$ is an element of $\R^{m \times |\mathcal{R}|}$, where $m = \sum_{C \in \mathcal{C}(G)} \prod_{v \in C} d_v$. By the rank-nullity theorem it holds that
\[
    \dim(\ker A_{G,\bfd}^{\top}) + \rk(A_{G,\bfd}^{\top}) = m.
\]
Since $P$ consists of a subset of columns of $A_{G,\bfd}$ and  since we have $\rk(P^{\top})= \rk(A_{G,\bfd}^{\top})$, it also holds that 
\[
    \dim(\ker P^{\top}) + \rk(A_{G,\bfd}^{\top}) = m
\]
Taking both equalities together, it follows that $\dim(\ker A_{G,\bfd}^{\top}) = \dim(\ker P^{\top})$. Now, fix a clique $S$ and a state $j_S \in \mathcal{R}_S$. Suppose that there is no column in $P$ that is indexed by $i \in \mathcal{R}$ such that $i_S = j_S$. We will show that $\ker P^{\top}$ then contains a vector $\bfx \in \R^m$ that is not in $\ker A_{G,\bfd}^{\top}$. Since $\ker A_{G,\bfd}^{\top} \subseteq \ker P^{\top}$, this implies the strict inequality $\dim(\ker A_{G,\bfd}^{\top}) < \dim(\ker P^{\top})$, which is a contradiction. \looseness=-1

To find a vector $\bfx$ in the kernel of $P^{\top}$, observe that there has to be a maximal clique $C$ in the graph $G$ such that $S \subseteq C$. We further fix some $j_{C\setminus S} \in \mathcal{R}_{C\setminus S}$. By the definition of $A_{G,\bfd}$, the row in $P$ that is indexed by $\theta_{j_S j_{C\setminus S}}^{(C)}$ must be equal to zero. Hence, we define $\bfx$ to be the unit vector that has zero entries everywhere but a one in the position indexed by $\theta_{j_S j_{C\setminus S}}^{(C)}$. Then, $\bfx^{\top}P = 0$ which is equivalent to $\bfx \in \ker P^{\top}$. To conclude, note that $\bfx$ is not in the kernel of $A_{G,\bfd}$ since $A_{G,\bfd}$ does not contain rows that are equal to zero.
\end{proof}

\begin{remark} \label{rem:LemmaNotTrue}
Lemma~\ref{lem:StatesInS} is no longer true in this generality if $S$ is not a clique. Consider the $3$-path $X_1 - X_2 - X_3$, where all variables are binary. In this case the matrix $A_{G, \bfd}$ has eight columns indexed by $i \in \mathcal{R}$. Now, let $P$ be the matrix obtained from removing the columns indexed by $(1,1,1)$ and $(1,2,1)$, that is,
\[ \small{
    P = \begin{blockarray}          {cp{0.2cm}p{0.2cm}p{0.2cm}p{0.2cm}p{0.2cm}p{0.2cm}}
        & 112 & 122 & 211 & 212 & 221 & 222 \\
        \begin{block}{c[cccccc]}
        11\cdot  & 1 & 0 & 0 & 0 & 0 & 0 \\
        12\cdot  & 0 & 1 & 0 & 0 & 0 & 0 \\
        21\cdot  & 0 & 0 & 1 & 1 & 0 & 0 \\
        22\cdot  & 0 & 0 & 0 & 0 & 1 & 1 \\
        \BAhline
        \cdot11  & 0 & 0 & 1 & 0 & 0 & 0 \\
        \cdot12  & 1 & 0 & 0 & 1 & 0 & 0 \\
        \cdot21  & 0 & 0 & 0 & 0 & 1 & 0 \\
        \cdot22  & 0 & 1 & 0 & 0 & 0 & 1 \\
        \end{block}
    \end{blockarray}}.
\]
Define $S$ to be the set $\{1,3\}$ and observe that the matrix $P$ does not contain a column indexed by $i$ such that $i_S=(1,1)$. However, it can be easily verified that the rank of $P$ is equal to $6$, which is equal to the rank of $A_{G, \bfd}$. 
\end{remark}

We now show that we can extend a separating hyperplane when we use the clique sum operation by simply adding zeros.

\begin{theorem} \label{thm:CliqueGluing}
Let  $G$ be a graph and let $\bfd \in \N^{|V(G)|}$ be the list of the number of states of the associated random variables at each node. Let $K$ be a complete graph and, similarly, let $\bfd^K \in \N^{|V(K)|}$ the list of the number of states of the random variables at each node. Let $S$ be a (potentially empty) clique in $G$.
Suppose that $\bfh$ is a separating hyperplane for $A_{G,\bfd}$ and let $\bfhbar$ be the hyperplane that is obtained from $\bfh$ by appending $\prod_{i \in S} \bfd_i \cdot \prod_{j \in V(K)} \bfd^K_{j}$ zeros to $\bfh$. Then $\bfhbar$ is a separating hyperplane for $A_{G +_S K, (\bfd,\bfd^{K})}$.
\end{theorem}
\begin{proof}
Let $S \neq \emptyset$. By Definition \ref{def:CliqueGluing}, the graph $G +_S K$ contains exactly the maximal cliques of $G$ plus the additional maximal clique $C_K = K \cup S$. By~\cite[Corollary 2.7]{hosten2002grobner},  this implies that 
\begin{align*}
    \rk(A_{G +_S K, (\bfd,\bfd^{K})}) &= \rk(A_{G,\bfd}) + (d_K - 1) + (d_K - 1) (d_S - 1) \\
    &=   \rk(A_{G,\bfd}) + d_S d_K - d_S,
\end{align*}
where $d_S := \prod_{i \in S} \bfd_i$ and $d_K := \prod_{j \in V(K)} \bfd^K_j$. Without loss of generality, let $P$ be the matrix consisting of the columns $\bfa_{i}$ of $A_{G,\bfd}$ such that $\bfh \cdot \bfa_i > 0$. Define $\mathcal{P}$ to be the set of column indices that appear in $P$, that is, $\mathcal{P} = \{i \in \mathcal{R}: \bfh \cdot \bfa_i > 0\}.$ Let $\overline{P}$ be the matrix consisting of columns $\bfa_{i,i^K}$ of $A_{G +_S K, (\bfd,\bfd^{K})}$ such that $\bfhbar \cdot \bfa_{i,i^K} > 0$, where $i^K \in \mathcal{R}^K := \prod_{i \in V(K)} [d_i]$ . By the definition of $\bfhbar$, the matrix  $\overline{P}$ contains exactly the columns $\bfa_{i,i^K}$ for all $i \in \mathcal{P}$ and  for all $i^K \in \mathcal{R}^K$. Thus, we can reorder the columns of $\overline{P}$ into different blocks $\overline{P}_{i}$ for each $i \in \mathcal{P}$. The blocks are given by 
\[
    \overline{P}_{i} = 
    \begin{blockarray}{c}
    \begin{block}{(c)}
        A_i \\
       \BAhline
       B_{(1, \ldots, 1),i}  \\
       \BAhline
       \vdots \\
        \BAhline
       B_{(\bfd_j)_{j \in S},i} \\
    \end{block}
    \end{blockarray}.
\]
Here, the matrix $A_i$ is given by $d_K = |\mathcal{R}^K|$ copies of the column $\bfa_i$, and the square matrices $B_{j_S,i} \in \mathbb{R}^{d_K \times d_K}$ are defined as the submatrices of $\overline{P}$ with rows indexed by $\theta_{j_S i^K}^{(C_K)}$ and columns indexed by $(i, l^K)$ for $i^K, l^K \in \mathcal{R}^K$. Observe that the matrix $B_{j_S,i} \in \mathbb{R}^{d_K \times d_K}$ is the identity matrix if $j_S=i_S$, and the zero matrix otherwise. 

To show that $\rk(\overline{P}) = \rk(A_{G +_S K, (\bfd,\bfd^{K})})$, we will choose columns of $\overline{P}$ such that the submatrix we obtain has the same rank as $A_{G +_S K, (\bfd,\bfd^{K})}$. To begin, we pick the first column of $\overline{P}_i$ for each $i \in \mathcal{P}$. Since the resulting submatrix contains $A_{G,\bfd}$, its rank is greater or equal than the rank of $A_{G,\bfd}$. Second, by Lemma~\ref{lem:StatesInS}, observe that there has to be an index $i \in \mathcal{P}$ such that $i_S = j_S$ for each $j_S \in \mathcal{R}_S$. Pick such an index $i \in \mathcal{R}$ for each $j_S \in \mathcal{R}_S$ and append the remaining $d_K-1$ columns of $\overline{P}_i$ to our submatrix. In other words, we append $d_S (d_K-1)$ columns to our submatrix. Each of the $d_K-1$ columns of the chosen matrices $\overline{P_i}$ contain a $(d_K - 1) \times (d_K - 1)$ identity matrix in a different position. Hence, the new $d_S (d_K - 1)$ columns are linearly independent. Since they are also linearly independent from the columns we chose in the first step, we conclude that the rank of our submatrix is at least $\rk(A_{G,\bfd}) + d_S (d_K - 1) = \rk(A_{G +_S K, (\bfd,\bfd^{K})})$. Since, $\overline{P}$ is a submatrix of $A_{G +_S K, (\bfd,\bfd^{K})}$, it follows that $\rk(\overline{P})= \rk(A_{G +_S K, (\bfd,\bfd^{K})})$. \\

It remains to consider $S = \emptyset$. In this case, $K$ is a maximal clique in $G +_S K$ and it holds that $\rk(A_{G +_S K, (\bfd,\bfd^{K})}) = \rk(A_{G,\bfd}) + d_K - 1$. With the same notation as above, the matrices $\overline{P_i}$ have the form
\[
     \overline{P}_{i} = 
    \begin{blockarray}{c}
    \begin{block}{(c)}
        A_i \\
       \BAhline
       I_{d_K} \\
    \end{block}
    \end{blockarray},
\]
where $I_{d_K}$ is the $d_K \times d_K$ identity matrix. We will again choose a submatrix of $\overline{P}$ with the same rank as $A_{G +_S K, (\bfd,\bfd^{K})}$. As before, we start by picking the first column of each $\overline{P}_{i}$ to obtain a submatrix of rank greater or equal than the rank of $A_{G,\bfd}$. Then, we choose an arbitrary index $i \in \mathcal{P}$ and append the remaining $d_K-1$ columns of $\overline{P}_i$ to our submatrix. These $d_K-1$ columns are linearly independent since they contain $I_{d_K-1}$. Moreover, these columns are linearly independent from the columns we chose in the first step. Hence, we conclude that $\rk(\overline{P}) = \rk(A_{G,\bfd}) + d_K - 1 = \rk(A_{G +_S K, (\bfd,\bfd^{K})})$.
\end{proof}

We now have all of the results required to prove the main theorem of this section.

\begin{theorem} \label{thm:ExpectedDimension}
    Let $G$ be a decomposable graph that is not a clique star. Then
    \[
        \dim(\text{Mixt}^2(\mathcal{M}_G)) = 2 \dim(\M_G) + 1.
    \] 
\end{theorem}
\begin{proof} Decomposable graphs are constructed by repeatedly performing clique-sums with maximal cliques. Let $G$ be a decomposable graph that is not a clique star. Let $\bfd \in \N^{|V(G)|}$ be an assignment of the number of states to the nodes of $G$ with $d_i \geq 2$ for all $i$. Since $G$ is decomposable, there exists an ordering $C_1,\dots,C_k$ of the maximal cliques of $G$ such that $G$ is obtained by performing the clique-sums in this order. Since $G$ is not a clique star, we may suppose further that the pairwise intersections of $C_1, C_2$ and $C_3$ are not all equal.

Let $H_{123}$ be the subgraph of $G$ with maximal cliques $C_1$, $C_2$ and $C_3$.  Let $\bfd_{123}$ be the restriction of $\bfd$ to the nodes in $C_1 \cup C_2 \cup C_3$. By Corollary~\ref{cor:ThreeCliques}, every non-clique star with three maximal cliques has a separating hyperplane. So there exists a separating hyperplane $\bfh_{123}$ for $A_{H_{123},\bfd_{123}}$. 

We proceed by induction on the number of maximal cliques. Let $H_{1\dots \ell}$ be the subgraph of $G$ with maximal cliques $C_1,\dots,C_\ell$. Let $\bfd_{1\dots \ell}$ be the restriction of $\bfd$ to the nodes in $C_1\cup \dots \cup C_{\ell}$. Suppose that $A_{H_{1\dots \ell}, \bfd_{1\dots \ell}}$ has a separating hyperplane $\bfh_{1\dots \ell}$. We perform a clique sum with $C_{\ell+1}$ to obtain $H_{1\dots \ell+1}$. By Theorem~\ref{thm:CliqueGluing}, there exists a hyperplane $\bfh_{1\dots \ell+1}$ that separates $A_{H_{1\dots \ell+1}, \bfd_{1\dots \ell+1}}$. Hence by induction, there exists a hyperplane $\bfh$ separating $A_{G,\bfd}$.

The rank of $A_{G,\bfd}$ is one more than the dimension of the graphical model specified by $G$ and $\bfd$. Hence by Theorem~\ref{thm:Draisma},
\[
\dim(\Mixt^2(\M_G)) = \dim(\Sec^2(\overline{\M_{G,\bfd}})) \geq 2 \dim(\M_{G,\bfd}) + 1.
\]
The right-hand side is the expected dimension of the secant variety, and hence is also an upper bound on the dimension of the secant. Hence they are equal.
\end{proof}

\section{Maximum Likelihood Degree}\label{sec:ml-degree}
For any statistical model $\M\subseteq \Delta_{n}$ and any data point $u$, \emph{maximum likelihood estimation} aims to find the distribution in the model under which the probability of observing this data is maximized. When $u$ is an independently identically distributed vector of counts, this amounts to solving the following optimization problem:
$$\text{Maximize $\ell_u(p):=\sum_{i=1}^nu_i\log(p_i)$ subject to $p\in\M$}.$$ The global maximizer of $\ell_u(p)$ over $\M$ is unique if it exists. It is called the \textit{maximum likelihood estimate (MLE)} of $u$. The algebraic complexity of solving this optimization problem is measured by the maximum likelihood degree.

\begin{definition}
The \textit{maximum likelihood degree (ML degree)} of an algebraic statistical model $\M$ is the number of complex critical points of $\ell_u(p)$ on the Zariski closure of $\M$ for generic data $u\in \C^n$.
\end{definition}

We will denote the ML degree of the model $\M$ by $\MLdeg(\M)$. For a thorough introduction to algebraic methods for maximum likelihood estimation, we refer the reader to Chapter 7 of~\cite{sullivant2018algebraic}. While there is a significant body of work on ML degrees of toric varieties, little is known about the ML-degrees of their mixtures. The current state of the art in this direction is~\cite{ml-degree-formula}, which studies the ML-degree problem for mixtures of independence models. In this section we give a recursive formula for the ML degree of mixtures of clique stars. We also obtain a closed-form formula for the ML degree of 2-mixtures of clique stars with only two cliques.

\begin{theorem}\label{thm:ml-deg-clique star}
    Let $G=(C_1 \cup \cdots \cup C_k \cup S, E)$ be a clique star. Let $\M_{C_1\indep\ldots\indep C_k}$ denote the independence model on $k$ random variables, whose state spaces have sizes $d_{C_1},\ldots, d_{C_k}$, respectively. Then \looseness=-1
    $$\MLdeg(\Mixt^r(\M_G))=\Big(\MLdeg(\Mixt^r(\M_{C_1\indep\ldots\indep C_k}))\Big)^{d_S}.$$ 
\end{theorem}

\begin{proof}
    Note that by Theorem~\ref{thm:clique star-ideals}, $\Mixt^r(\M_G)$ is the Cartesian product of mixtures of independence models, i.e. $\Mixt^r(\M_G)=\prod_{j_S\in\mathcal{R}_S}\Mixt^r(\M_{j_S; C_1\indep\ldots\indep C_k})$. Hence, the log-likelihood function $\ell_u$ is maximized at $p=\prod_{j_S\in\mathcal{R}_S}p_{j_S}$ whenever it is maximized at the restriction to each component $p_{j_S}$. We conclude that the number of critical points of the log-likelihood function is the product of the ML degrees with respect to each component.
\end{proof}
\begin{example}
Let $G$ be the star graph on the right of Figure \ref{fig:ml-degree-clique stars} where the leaves consist of two binary and one ternary random variable and the random variable at the central node has $k$ states. The 2-mixture of the independence model of two binary and one ternary random variables has ML degree 26. By Theorem~\ref{thm:ml-deg-clique star}, the model $\Mixt^2(\M_G)$ has ML degree $26^k$ for any $k\in\mathbb{N}$.
\end{example}
Whenever $G=(C_1\cup C_2\cup S, E)$ and $r=2$, a closed-form formula for the ML degree may be obtained using the results in~\cite{ml-degree-formula}. When $d_{C_1}=3$, we obtain the following corollary.
\begin{corollary}
    Let $G=(C_1\cup C_2\cup S, E)$ be a clique star and assume $d_{C_1}=3$. Then $$\MLdeg(\Mixt^2(\M_G))=(2^{d_{C_2}+1}-6)^{d_S}.$$
\end{corollary}
\begin{proof}
By~\cite[Theorem 3.12]{ml-degree-formula}, the ML degree of $\Mixt^2(\M_{C_1\indep C_2})$ where $d_{C_1}=3$ is $2^{d_{C_2}+1}-6$. The result follows
    directly from Theorem~\ref{thm:ml-deg-clique star}.
\end{proof}

A closed-form formula for the ML degree may still be obtained for $d_{C_1}>3$; see~\cite[Section 4.2]{ml-degree-formula}.
    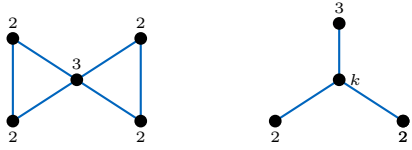
\begin{figure}
        \centering

\tikzset{
      every node/.style={},
}
\begin{tikzpicture}[align=center]
    \draw[MyBlue] (0,0) -- (-0.85,0.55);
    \draw[MyBlue] (0,0)  -- (-0.85,-0.55);
    \draw[MyBlue] (0,0) -- (0.85,0.55);
    \draw[MyBlue] (0,0) -- (0.85,-0.55);
    \draw[MyBlue] (-0.85,0.55) -- (-0.85,-0.55);
    \draw[MyBlue] (0.85,0.55) -- (0.85,-0.55);
    
    \filldraw (0,0) circle (2pt) node[anchor=south] (1) {\tiny $3$};
    \filldraw (-0.85,0.55) circle (2pt) node[anchor=south] (3) {\tiny $2$};
    \filldraw (-0.85,-0.55) circle (2pt) node[anchor=north] {\tiny $2$};
    \filldraw (0.85,0.55) circle (2pt) node[anchor=south] {\tiny $2$};
    \filldraw (0.85,-0.55) circle (2pt) node[anchor=north] {\tiny $2$};
\end{tikzpicture}
\hspace{3em}
\begin{tikzpicture}[align=center]
 
    \draw[MyBlue] (0,0)  -- (-0.85,-0.55);
    \draw[MyBlue] (0,0) -- (0.85,-0.55);
     \draw[MyBlue] (0,0) -- (0,0.75);
    
    \filldraw (0,0) circle (2pt) node[anchor=west] (1) {\tiny $k$};
    \filldraw (-0.85,-0.55) circle (2pt) node[anchor=north] {\tiny $2$};
    \filldraw (0.85,-0.55) circle (2pt) node[anchor=north] {\tiny $2$};
    \filldraw (0.85,-0.55) circle (2pt) node[anchor=north] {\tiny $2$};
    \filldraw (0,0.75) circle (2pt) node[anchor=south] {\tiny $3$};
\end{tikzpicture}

\caption{Two star-cliques.}
\label{fig:ml-degree-clique stars}
\end{figure}
\begin{example}
 Consider the model $\M_G$ given by the graph on the left of Figure \ref{fig:ml-degree-clique stars}.
The model $\M_{C_1\indep C_2}$ is the independence model of two random variables with four states. Its 2-mixture has ML degree 191, according to~\cite[Section 4.2]{ml-degree-formula}. By Theorem~\ref{thm:ml-deg-clique star}, the ML degree of $\Mixt^2(\M_G)$ is $191^3=6,967,871$. 
\end{example}

\section{Discussion}
In this paper, we derived novel algebro-geometric results for mixtures of discrete graphical models. Our work opens up some natural questions for further studies. \\

\noindent
\textbf{Dimension of non-decomposable graphs.} In all our computations, we have found that the second secant of \emph{any} graphical model that is not a clique star also has expected dimension. Hence, it would be desirable to extend our proof in Section \ref{sec:dimension} to non-decomposable graphs. We conjecture that extending the hyperplanes by adding zeros as in Theorem~\ref{thm:CliqueGluing} also works if one glues to a set $S$ that is not a clique. However, as discussed in Remark~\ref{rem:LemmaNotTrue}, our current proof technique is not applicable since Lemma~\ref{lem:StatesInS} is only true if $S$ is a  clique. That being said, the operations described in Theorems \ref{thm:AddBinaryRV} and \ref{thm:IncreaseStates} do not rely on decomposability and Theorem \ref{thm:AddBinaryRV} in particular may be used to obtain some indecomposable graphs. So the second mixture associated to any graph that can be built in this way will also have expected dimension.  \\

\noindent
\textbf{Higher-order secants.} 
To prove that the dimension of second secants of decomposable graphs that are not clique-stars is as expected, we made use of the tool developed by~\cite{draisma} stated in Theorem~\ref{thm:Draisma}. It allowed us to obtain lower bounds on the dimension of second secants by constructing separating hyperplanes. The original statement of Corollary 2.3 in~\cite{draisma} also allows higher order secants with $r > 2$. In this case, one may find $r-1$ hyperplanes and use them to split the vertices of $\text{conv}(A)$ into disjoint sets so that the corresponding columns of each set have full rank; also see~\cite[Proposition 2.6]{banos2019dimensions}. However, there could be more defective secants than the clique star, and in addition, one certainly has to consider more base cases for the induction than the two binary models corresponding to the graphs $P_4$ and $P_3 \sqcup P_1$.  \\

\noindent
\textbf{Ideals of secants of non-clique-stars.} While we fully described the ideal $I_G^{(r)}$ of secants of clique-stars in terms of ideals obtained from tensors of rank $r$ in Theorem~\ref{thm:clique star-ideals}, it remains an open question to determine the ideals of secants obtained from graphs that are not clique stars. Is it possible to characterize the graphs whose ideal $I_G^{(r)}$ is generated only by minors corresponding to the Markov property? How can we explain the other polynomials that arose in Examples~\ref{ex:4-path}~and~\ref{ex:binary-path}?  \\

\noindent
\textbf{Joins.} In this work, we considered secant varieties that correspond to mixtures of the \emph{same} graphical model.  It is also common to consider mixtures of multiple models corresponding to \emph{different} graphs. Algebraically, this leads to the study of joins instead of secants.

\section*{Acknowledgments}
We thank Kaie Kubjas for suggesting this problem and many helpful insights. We also thank Benjamin Hollering, Elizabeth Gross, Joseph Cummings, and Bernd Sturmfels for helpful discussions. We are thankful to Marina Garrote-L\'opez and Nataliia Kushnerchuk for their comments on the manuscript. Part of this research was performed while the authors were visiting the Institute for Mathematical and Statistical Innovation (IMSI), which is supported by the National Science Foundation (Grant No. DMS-1929348). The project has received funding from the European Research Council (ERC) under the European Union’s Horizon 2020 research and innovation programme (grant agreement No 883818). Yulia Alexandr was supported by the National Science Foundation Graduate Research Fellowship under Grant No. DGE 2146752. Jane Ivy Coons was supported by the L'Or\'eal-UNESCO For Women in Science UK and Ireland Rising Talent Award Programme. Nils Sturma acknowledges support by the Munich Data Science Institute at Technical University of Munich via the Linde/MDSI PhD Fellowship program.

\bibliographystyle{plain}
\bibliography{sample}
\end{document}